\documentclass[11pt]{article}

\usepackage[T1]{fontenc}
\usepackage[a4paper,margin=1in]{geometry}
\usepackage{amsmath,amssymb,amsthm,mathtools,mathrsfs}
\usepackage{tikz}
\usepackage{microtype}
\usepackage[hidelinks]{hyperref}
\AtBeginEnvironment{thebibliography}{\sloppy}

\newtheorem{theorem}{Theorem}[section]
\newtheorem{proposition}[theorem]{Proposition}
\newtheorem{lemma}[theorem]{Lemma}
\newtheorem{remark}[theorem]{Remark}
\newtheorem{introtheorem}[theorem]{Theorem}
\newtheorem{introproposition}[theorem]{Proposition}

\newcommand{\R}{\mathbb R}
\newcommand{\Hh}{\mathbb H}
\newcommand{\Ss}{\mathbb S}
\newcommand{\Ecal}{\mathcal E_{\Lambda_0}}
\newcommand{\Hcal}{\mathcal H_{\Lambda_0}(M)}
\newcommand{\Hlam}{\mathcal H_\lambda(M)}
\newcommand{\Slam}{S_{\lambda,k}(M)}
\newcommand{\Sth}{S_{\Lambda_0,k}(M)}
\newcommand{\dd}{\,d}

\title{Critical GJMS Equations on \(\mathbb H^n\times\mathbb S^m\)}
\author{Q. Hua \and J. Li\thanks{Corresponding author. Email address:
\href{mailto:jungangli@ustc.edu.cn}{\nolinkurl{jungangli@ustc.edu.cn}}.}
\and C. Tao}
\date{}

\begin{document}
\maketitle

\begin{abstract}
Let \(M=\mathbb H^n\times\mathbb S^m\), where \(n\ge2\), \(m\ge1\),
and \(N=n+m\).  Let \(P_k\) be the order-\(2k\) GJMS operator, with
\(1\le k<N/2\), and assume that
\(\Lambda_0=\inf\sigma_{L^2(M)}(P_k)>0\).  We study
\[
 P_kU-\lambda U=|U|^{q-2}U,
 \qquad q=\frac{2N}{N-2k},
 \qquad 0<\lambda\le\Lambda_0,
\]
and the attainment of the associated critical quotient
\(S_{\lambda,k}(M)\).

Let \(S_{N,k}\) be the Euclidean best Sobolev constant.  For
\(0<\lambda<\Lambda_0\), the inequality
\(S_{\lambda,k}(M)<S_{N,k}\) implies attainment and a nontrivial weak
solution.  Localized Euclidean extremals establish this inequality when
\[
 N\ge4k,
 \qquad\text{or}\qquad
 2k+2\le N<4k\quad\text{and}\quad
 \lambda>\Lambda_{\mathrm{loc}},
\]
where \(\Lambda_{\mathrm{loc}}\) is explicit.  If \(N\ge2k+2\) and
\(S_{\Lambda_0,k}(M)<S_{N,k}\), attainment also holds at
\(\lambda=\Lambda_0\) in the threshold form completion.

At the threshold, \(L^2\)-coercivity fails precisely on the constant
spherical eigenspace.  We combine cocompactness for its hyperbolic
coefficient with a profile decomposition relative to the critical
transformations preserving
\(\mathcal A=\mathbb R^{n-1}\times\{0\}\).  Under the threshold hypotheses
above, the strict Euclidean inequality excludes concentration escaping
\(\mathcal A\) from normalized minimizing sequences.  If
\(r_j^{(J)}\) denotes the remainder after the first \(J\) extracted
profiles, then
\[
 \lim_{J\to\infty}\limsup_{j\to\infty}
 \|r_j^{(J)}\|_{L^q(\mathbb R^N)}=0,
\]
which yields compactness modulo the axis-preserving transformations.
\end{abstract}



\section{Introduction}

Let \(P_k\) be the order-\(2k\) GJMS operator
\cite{GrahamJenneMasonSparling1992} on
\[
 M=\Hh^n\times\Ss^m,
\]
where the factors have sectional curvatures \(-1\) and \(+1\),
respectively.  We consider
\begin{equation}\label{eq:intro-equation}
 P_kU-\lambda U=|U|^{q-2}U,
 \qquad
 q=\frac{2N}{N-2k},
 \qquad N=n+m.
\end{equation}
We use the self-adjoint \(L^2(M)\)-realization of \(P_k\), with domain
\(\mathcal D(P_k)\).  Its spectrum is characterized by the quotient
\[
 \sigma_{L^2(M)}(P_k)
 =
 \left\{
 \mu\in\R:
 \inf_{u\in\mathcal D(P_k)\setminus\{0\}}
 \frac{\|(P_k-\mu)u\|_{L^2(M)}}{\|u\|_{L^2(M)}}=0
 \right\}.
\]
We assume that the bottom of this spectrum is positive,
\[
 \Lambda_0:=\inf\sigma_{L^2(M)}(P_k)>0,
\]
and treat both \(0<\lambda<\Lambda_0\) and
\(\lambda=\Lambda_0\).

The product geometry has a direct origin in the singular Yamabe problem.
Write \(y=r\theta\), where \(r>0\) and \(\theta\in\Ss^m\).  Then
\begin{equation}\label{eq:intro-singular-yamabe-product}
 |y|^{-2}(\dd x^2+\dd y^2)
 =
 \frac{\dd x^2+\dd r^2}{r^2}+g_{\Ss^m}
 =
 g_{\Hh^n}+g_{\Ss^m}.
\end{equation}
After conformal compactification, this is the standard identification
\[
 \Ss^N\setminus\Ss^{n-1}
 \cong
 \Hh^n\times\Ss^m.
\]
The product metric is therefore a complete conformal metric on the
complement of a round positive-dimensional singular set.  Its scalar
curvature is
\[
 R_{g_M}
 =
 m(m-1)-n(n-1)
 =
 (m-n)(N-1),
\]
whose sign is not fixed.

Loewner and Nirenberg initiated the construction of complete conformal
metrics of constant negative scalar curvature through boundary blow-up
solutions of the Yamabe equation
\cite{LoewnerNirenberg1974ConformalPDE}.  Schoen and Yau related complete
conformal metrics on domains in the sphere to distributional Yamabe
equations.  Under the bounded-curvature hypothesis of their Theorem~5.1, a
complete conformal metric of scalar curvature \(1\) gives a global weak
solution; their examples include round singular sets arising from products
with compact hyperbolic manifolds
\cite{SchoenYau1988ConformallyFlat}.
Mazzeo developed the asymptotic regularity theory when the singular set is a
smooth submanifold \cite{Mazzeo1991SingularYamabe}.  For the round
subspace appearing above, Bettiol and Piccione used the same conformal
identification to construct infinitely many periodic singular Yamabe
solutions in the range \(n<m\)
\cite{BettiolPiccione2018NoncompactYamabe}.

For \(k=1\) and \(\lambda=0\), equation
\eqref{eq:intro-equation} is, up to the normalization of its right-hand
side, the conformal scalar-curvature equation on this complete singular
background.  For \(k>1\), it is the corresponding higher-order conformal
equation associated with the GJMS operator.  The term \(-\lambda U\) is a
Brézis--Nirenberg spectral perturbation, so a solution with \(\lambda>0\)
should not itself be identified with an unperturbed singular Yamabe metric.
The product nevertheless models compactness and concentration of conformal
factors near a positive-dimensional singular stratum.  The threshold profile
decomposition below is relevant to perturbative geometric problems near such
a stratum; no singular Yamabe existence theorem is asserted here.

The model for \eqref{eq:intro-equation} is the Euclidean
Brézis--Nirenberg problem.  On a bounded domain
\(\Omega\subset\R^d\), the equation
\[
 -\Delta u-\lambda u=|u|^{2^*-2}u,
 \qquad
 u|_{\partial\Omega}=0,
 \qquad
 2^*=\frac{2d}{d-2},
\]
combines a lower-order spectral perturbation with the loss of compactness at
the critical Sobolev exponent.  The determination of the Euclidean best
constant and its extremals goes back to Aubin and Talenti
\cite{Aubin1976Sobolev,Talenti1976BestConstant}.  Brézis and Nirenberg
\cite{BrezisNirenberg1983} showed that, in dimensions \(d\ge4\), every
\(0<\lambda<\lambda_1(\Omega)\) lowers the variational quotient strictly
below the Euclidean Sobolev level; dimension three exhibits a genuine
lower-dimensional obstruction.  The same competition between the cutoff
error of a concentrating extremal and the negative lower-order term persists
for polyharmonic operators.  The works of Pucci--Serrin,
Edmunds--Fortunato--Jannelli, and Gazzola
\cite{PucciSerrin1990CriticalDimensions,
EdmundsFortunatoJannelli1990Biharmonic,Gazzola1998CriticalGrowth}
show how the order of the operator changes the critical dimension and the
parameter ranges in which this comparison succeeds.

Hyperbolic space introduces a different spectral geometry.  Its Laplacian
has continuous spectrum, the isometry group is noncompact, and mass may
escape to infinity without concentrating in a fixed coordinate chart.
For the second-order critical equation on \(\Hh^n\), Mancini and Sandeep
\cite{ManciniSandeep2008SemilinearHn} classified positive finite-energy
solutions.  In their parameterization,
\[
 -\Delta_{\Hh^n}u-\mu_{\mathrm{MS}}u=u^{(n+2)/(n-2)}
\]
such solutions exist precisely for
\[
 \frac{n(n-2)}4<\mu_{\mathrm{MS}}\le\frac{(n-1)^2}{4}
 \qquad(n\ge4),
\]
and are unique up to hyperbolic isometries; in dimension \(n=3\), no positive
finite-energy solution exists.  At the upper endpoint, finite energy is
understood with respect to the completion defined by the threshold quadratic
form.  Thus, for
\(P_1=-\Delta_{\Hh^n}-n(n-2)/4\) and
\(\lambda=\mu_{\mathrm{MS}}-n(n-2)/4\), this range is
\(0<\lambda\le1/4\), including the spectral bottom of \(P_1\).
Bhakta and Sandeep subsequently proved a Struwe-type Palais--Smale
decomposition for the second-order Poincaré--Sobolev functional on
\(\Hh^n\), separating escape by hyperbolic isometries from localized
Euclidean concentration at the critical exponent
\cite{BhaktaSandeep2012PoincareSobolev}.  Bhakta, Ganguly, Karmakar, and
Mazumdar later proved a linear quantitative stability estimate for finite
sums of weakly interacting hyperbolic bubbles in the same second-order
problem.  For \(3\le n\le5\) and \(p>2\), under their parameter and
proximity hypotheses, the \(H^1\)-distance to such a sum is controlled by
the \(H^{-1}\)-norm of the Euler--Lagrange residual; they also showed that
the analogous linear multi-bubble estimate fails for \(1<p\le2\)
\cite{BhaktaGangulyKarmakarMazumdar2025QuantitativeStruwe}.  For higher order equations,
Lu and Yang developed the relevant Paneitz and GJMS Green kernels and
Hardy--Sobolev--Maz'ya inequalities
\cite{LuYang2019PaneitzHSM,LuYang2022GreenGJMS}, while Li, Lu, and Yang
\cite{LiLuYang2022HigherOrderBN} established higher-order
Brézis--Nirenberg results on bounded hyperbolic domains and on the whole
space, together with nonexistence and symmetry results.  Their whole-space
higher-order existence theorem lies strictly below the corresponding
spectral edge.  A recent preprint of Lu and Tao proves attainment in the
critical higher-order Hardy--Sobolev--Maz'ya inequality on the upper
half-space for \(k\ge2\) and \(n\ge2k+2\).  As a consequence, they obtain a
positive solution of the pure-hyperbolic GJMS equation at
\(\alpha=\prod_{i=1}^k(2i-1)^2/4\)
\cite{LuTao2026HighOrderHSM}.  The product geometry considered here
introduces an additional discrete spherical spectrum and a second
concentration regime relative to the singular axis.

The harmonic analysis underlying these developments begins with the
Helgason--Fourier transform, inversion, and Plancherel theory on noncompact
symmetric spaces \cite{Helgason2000GroupsGeometricAnalysis}.  Anker's
\(L^p\) multiplier theorem and sharp estimates for functions of the
Laplacian \cite{Anker1990LpMultipliers,Anker1992SharpLaplacian}, followed by
the heat-kernel and Green-function estimates of Anker and Ji
 \cite{AnkerJi1999HeatGreen}, provide a global calculus adapted to exponential
volume growth.  These results provide the spectral calculus used below.  Set
\[
 A=\sqrt{-\Delta_{\Hh^n}-(n-1)^2/4},
\]
let \(E_A\) be its spectral resolution, and write
\[
 dE_A(\rho)=\mathsf e_A(\rho)\,\dd\rho.
\]
For \(2<s<\infty\), put \(s'=s/(s-1)\).  At hyperbolic frequencies
\(0<\rho\le1\), the estimate of Chen and Hassell
\cite{ChenHassell2018SpectralMeasure} gives
\[
 \|\mathsf e_A(\rho)\|_{L^{s'}(\Hh^n)\to L^s(\Hh^n)}
 \le C_{n,s}\rho^2,
 \qquad 0<\rho\le1.
\]
The factor \(\rho^2\) quantifies the vanishing of this measure at the bottom
of the real hyperbolic spectrum and is used in the low-frequency estimate of
Lemma~\ref{lem:low-frequency}.

The product \(M=\Hh^n\times\Ss^m\) couples this continuous hyperbolic
spectrum to the discrete spherical spectrum.  This coupling is not a
compact perturbation of the pure hyperbolic problem.  If
\[
 \mathscr H_\ell(\Ss^m)
 =
 \ker\!\left(
  -\Delta_{\Ss^m}-\ell(\ell+m-1)
 \right),
\]
then the restriction of \(P_k\) to
\(L^2(\Hh^n)\otimes\mathscr H_\ell(\Ss^m)\) is described by the scalar
multiplier
\[
 p_{k,\ell}(\rho)
 =
 \prod_{r=1}^k
 \left[
  \rho^2+
  \bigl(\ell+\tfrac{m-1}{2}+k+1-2r\bigr)^2
 \right].
\]
Consequently, the spectral bottom is the minimum of
\(p_{k,\ell}(0)\) over \(\ell\ge0\).  The scattering calculation in
Appendix~\ref{app:scattering} shows that the assumption \(\Lambda_0>0\)
forces
\[
 p_{k,0}(0)<p_{k,1}(0)<p_{k,2}(0)<\cdots.
\]
Thus the spherical subspace associated with the spectral bottom is
\[
 \mathscr Y_0
 =
 \mathscr H_0(\Ss^m)
 =
 \operatorname{span}\{1\}.
\]
Accordingly, we decompose \(L^2(M)\) into this subspace and its orthogonal
complement.  On
\(L^2(\Hh^n)\otimes\mathscr Y_0\), the threshold form loses its
\(L^2\)-control at hyperbolic frequency \(\rho=0\); on the orthogonal
spherical-harmonic complement, it remains uniformly coercive in \(H^k(M)\).
This simultaneous low-frequency degeneracy and critical concentration is
the principal analytic difficulty created by the product geometry.

The geometry becomes especially transparent in upper half-space
coordinates.  Conformal covariance transforms the quadratic form into
\begin{equation}\label{eq:intro-conformal-form}
 \int_{\R^N}(-\Delta)^kw\,w\,\dd z
 -
 \lambda\int_{\R^N}\frac{w^2}{|y|^{2k}}\,\dd z,
 \qquad
 z=(x,y)\in\R^{n-1}\times\R^{m+1},
 \qquad
 w\in C_c^\infty(\R^{n+m}\setminus\R^{n-1}).
\end{equation}
The singular set is the positive-dimensional subspace
\[
 \mathcal A=\R^{n-1}\times\{0\}.
\]
The positivity of the threshold form follows from a
Hardy--Sobolev--Maz'ya inequality of Yang
\cite{Yang2021PolyharmonicHSM}.  More
precisely, Yang's theorem is formulated on the complement of an affine
subspace in Euclidean space.  For \(k\ge2\), specialize the dimensions to
\(n_Y=n-1\) and \(m_Y=m+1\), and take the critical exponent \(p=q\).
Indeed, \(n_Y\ge1\), \(m_Y\ge2\), \(2\le k<N/2\), and
\(q=2N/(N-2k)\) is the upper endpoint in Yang's admissible range.
Consequently, in the range used in
Proposition~\ref{prop:positive-threshold}, there is a constant
\(C=C(n,m,k)>0\) such that
\begin{equation}\label{eq:intro-yang-hsm}
 \int_{\R^N}(-\Delta)^kw\,w\,\dd z
 -c_{m+1,k}\int_{\R^N}\frac{|w|^2}{|y|^{2k}}\,\dd z
 \ge
 C\left(\int_{\R^N}|w|^q\,\dd z\right)^{2/q}
\end{equation}
for every
\(w\in C_c^\infty(\R^N\setminus\R^{n-1})\), where
\begin{equation}\label{eq:yang-hardy-coefficient}
 c_{m+1,k}
 =
 \min_{\ell\ge0}
 \prod_{i=0}^{k-1}
 \frac{(m+1+2\ell-2k+4i)^2}{4}.
\end{equation}
The change of index \(i=k-j\) in
\eqref{eq:yang-hardy-coefficient} shows that
\[
 c_{m+1,k}
 =
 \min_{\ell\ge0}
 \prod_{j=1}^k
 \left(\ell+\frac{m-1}{2}+k+1-2j\right)^2
 =\Lambda_0.
\]
Appendix~\ref{app:scattering} gives an independent geometric derivation of
this sharp Hardy coefficient from the scattering operator of
\(\Hh^{N+1}\).  The positive \(L^q\)-remainder in
\eqref{eq:intro-yang-hsm}, however, is the analytic conclusion supplied by
Yang's theorem and is not a consequence of the spectral computation alone.
For \(k=1\), Maz'ya's cylindrical Hardy--Sobolev inequality
\cite{Mazya2011SobolevSpaces}
gives \eqref{eq:intro-yang-hsm} with the coefficient
\((m-1)^2/4=\Lambda_0\).  Thus these inequalities imply the
positive critical embedding at the spectral bottom.  The compactness and
attainment statements at that value of the parameter require the additional
arguments developed in this paper.

Consider a critical Euclidean concentration centered at
\((\xi_j,\eta_j)\) with scale \(\mu_j^{-1}\).  If
\[
 \mu_j|\eta_j|=O(1),
\]
the rescaled limit still sees \(\mathcal A\) and the Hardy term in
\eqref{eq:intro-conformal-form}.  Such concentration is described by
translations parallel to \(\mathcal A\) and simultaneous dilations in all
variables.  If instead
\[
 \mu_j|\eta_j|\longrightarrow\infty,
\]
the singular set leaves every fixed rescaled compact set, and the limiting
quadratic form is the ordinary Euclidean \(\dot H^k\)-energy.  These are the
two geometrically distinct ways in which a critical sequence can escape.

This dichotomy belongs to a broader compactness theory.  Lions'
concentration--compactness principle
\cite{Lions1984ConcentrationCompactness,Lions1985LimitCase} and Struwe's
global compactness theorem \cite{Struwe1984GlobalCompactness} isolate
critical Euclidean bubbles in variational sequences.  Solimini and Gérard
\cite{Solimini1995LorentzCompactness,Gerard1998SobolevDefect} formulated
profile decompositions for bounded Euclidean Sobolev sequences.  The higher
differential order alone does not distinguish the present result from that
theory: Gérard's theorem with \(s=k\) already gives a
translation--dilation decomposition for arbitrary bounded sequences in
\(\dot H^k(\R^N)\) \cite{Gerard1998SobolevDefect}.

For a point singularity and a strictly subcritical second-order Hardy
coefficient, Cao and Peng
\cite{CaoPeng2003SingularGlobalCompactness}
decomposed Palais--Smale sequences into singular bubbles at the origin and
free Euclidean bubbles satisfying the scale--distance condition
\(\kappa_j|y_j|\to\infty\).  Bhakta and Sandeep considered Palais--Smale
sequences for a coercive second-order Poincaré--Sobolev functional strictly
below the hyperbolic spectral bottom and obtained hyperbolic translates
together with localized Euclidean bubbles
\cite{BhaktaSandeep2012PoincareSobolev}.  Their
Hardy--Sobolev--Maz'ya interpretation is made in a cylindrically symmetric
class.  Sandeep and Tintarev
\cite{SandeepTintarev2019ProfileManifolds}
treated arbitrary bounded \(H^{1,2}\)-sequences on complete manifolds of
bounded geometry, separating profiles on manifolds at infinity from local
Euclidean concentrations.

The threshold form in this paper is not equivalent to the standard
\(H^k(M)\)-norm.  On the constant spherical eigenspace,
\[
 p_{k,0}(\rho)-\Lambda_0
 \asymp
 \rho^2(1+\rho^2)^{k-1},
\]
whereas on its orthogonal spherical-harmonic complement the multiplier is
bounded below by a positive multiple of
\((1+\rho^2+\ell^2)^k\).  The first component must therefore be treated in
the homogeneous space
\[
 X_k(\Hh^n)
 =
 \overline{C_c^\infty(\Hh^n)}^{\,\|\cdot\|_{X_k}},
 \qquad
 \|f\|_{X_k}
 =
 \bigl\|A(1+A^2)^{(k-1)/2}f\bigr\|_{L^2(\Hh^n)},
\]
while the second is controlled in
\(H^k(M)\).  This threshold spectral separation is absent from the cited
second-order decompositions.

There is also a change of limiting operator that is measured by
\[
 \mu_j\operatorname{dist}\bigl((\xi_j,\eta_j),\mathcal A\bigr)
 =
 \mu_j|\eta_j|.
\]
Here
\[
 G_{\mathcal A}
 =
 \left\{
 \mathcal T_{b,\sigma}:
 (\mathcal T_{b,\sigma}w)(x,y)
 =
 \sigma^{(N-2k)/2}w\bigl(\sigma(x-b),\sigma y\bigr),
 \qquad
 b\in\R^{n-1},\ \sigma>0
 \right\}
\]
is the group of critical translations and dilations preserving
\(\mathcal A\).
When this quantity remains bounded, the singular axis survives the
rescaling and the concentration is represented by the exact
axis-preserving group \(G_{\mathcal A}\).  When it tends to infinity, the
Hardy potential disappears from every fixed rescaled compact set and the
limiting energy is the free \(\dot H^k(\R^N)\)-energy.  Such a concentration is
not another \(G_{\mathcal A}\)-profile: it belongs to a different limiting
quadratic form.  Theorem~\ref{thm:cylindrical-profile} gives the
\(G_{\mathcal A}\)-profile decomposition after Euclidean concentration far
from the singular axis has been excluded, and
Proposition~\ref{prop:packet-exclusion} proves that
the strict inequality \(S_{\Lambda_0,k}(M)<S_{N,k}\) gives this exclusion
for normalized minimizing sequences.

\subsection*{Main results}

For \(0<\lambda\le\Lambda_0\) and
\(u,v\in C_c^\infty(M)\), let
\[
 Q_\lambda(u,v)
 =
 \int_M(P_ku)v\,\dd V
 -\lambda\int_Muv\,\dd V,
 \qquad
 Q_\lambda(u)=Q_\lambda(u,u),
\]
and define
\[
 \mathcal H_\lambda(M)
 =
 \overline{C_c^\infty(M)}^{\,Q_\lambda^{1/2}},
 \qquad
 S_{\lambda,k}(M)
 =
 \inf_{u\in C_c^\infty(M)\setminus\{0\}}
 \frac{Q_\lambda(u)}{\|u\|_{L^q(M)}^2}.
\]
We denote by
\[
 S_{N,k}
 =
 \inf_{\varphi\in C_c^\infty(\R^N)\setminus\{0\}}
 \frac{\|\varphi\|_{\dot H^k(\R^N)}^2}
      {\|\varphi\|_{L^q(\R^N)}^2}
\]
the Euclidean sharp Sobolev constant.  Throughout the following statements,
\[
 n\ge2,
 \qquad m\ge1,
 \qquad N=n+m,
 \qquad 1\le k<\frac N2,
 \qquad \Lambda_0>0.
\]

In the range \(2k+2\le N<4k\), the lower-dimensional strict-gap condition
involves an explicit coefficient.  Let \({}_2F_1\) denote the Gauss
hypergeometric function, and set
\begin{equation}
 \mathcal A_{N,k}^{\mathrm{loc}}
 =
 \frac{2^{2k}|\Ss^{N-1}|\Gamma(k+1)^2}{N-2k}
 =
 \frac{2^{2k+1}\pi^{N/2}\Gamma(k+1)^2}
 {(N-2k)\Gamma(N/2)},
 \label{eq:A-local}
\end{equation}
\begin{equation}
\begin{aligned}
 \mathcal F_{n,m,k}
 &=
 \frac{2^{m+1}\pi^{(N-1)/2}\Gamma(m/2)}
 {\Gamma((n-1)/2)\Gamma(m)}
 \int_0^1\int_0^{\sqrt{1-t^2}}
 s^{n-2}t^m(1+t)^{-2k}(1-s^2-t^2)^{2k}
 \\
 &\quad\times
 {}_2F_1\!\left(\frac N2,1;k+1;1-s^2-t^2\right)^2
 {}_2F_1\!\left(k,\frac m2;m;\frac{4t}{(1+t)^2}\right)
 \,\dd s\,\dd t,
\end{aligned}
\label{eq:F-local}
\end{equation}
and
\begin{equation}
 \Lambda_{\mathrm{loc}}
 =
 \frac{\mathcal A_{N,k}^{\mathrm{loc}}}{\mathcal F_{n,m,k}}.
 \label{eq:lambda-local}
\end{equation}

\begin{introtheorem}[Compactness below the spectral bottom]
\label{thm:intro-subthreshold-compactness}
Assume
\[
 0<\lambda<\Lambda_0,
 \qquad
 S_{\lambda,k}(M)<S_{N,k}.
\]
Then \(S_{\lambda,k}(M)\) is attained by a nonzero
\(u\in\mathcal H_\lambda(M)=H^k(M)\).  If
\(\|u\|_{L^q(M)}=1\), then
\[
 Q_\lambda(u)=S_{\lambda,k}(M)
\]
and
\[
 Q_\lambda(u,\phi)
 =
 S_{\lambda,k}(M)
 \int_M|u|^{q-2}u\phi\,\dd V
 \qquad
 \bigl(\phi\in\mathcal H_\lambda(M)\bigr).
\]
Consequently,
\[
 U=S_{\lambda,k}(M)^{1/(q-2)}u
\]
is a nontrivial weak solution of
\[
 P_kU-\lambda U=|U|^{q-2}U
\]
in \(\mathcal H_\lambda(M)^*\), where
\[
 \mathcal H_\lambda(M)^*
 =
 \left\{F:\mathcal H_\lambda(M)\to\mathbb R:
 F\ \text{is linear and continuous}\right\}.
\]
\end{introtheorem}

This statement is proved in
Theorem~\ref{thm:subthreshold-compactness}.

The two localized-extremal calculations used to verify the strict gap can be
stated together as follows.

\begin{introproposition}[Strict gap from localized Euclidean extremals]
\label{prop:intro-localized-strict-gap}
Let \(0<\lambda\le\Lambda_0\).  Then
\[
 S_{\lambda,k}(M)<S_{N,k}
\]
in either of the following cases:
\[
 N\ge4k,
\]
or
\[
 2k+2\le N<4k,
 \qquad
 \lambda>\Lambda_{\mathrm{loc}}.
\]
No conclusion is asserted by these localized constructions when
\(N=2k+1\).
\end{introproposition}

This proposition combines Lemma~\ref{lem:strict-high-dim} and
Proposition~\ref{prop:strict-low-dim}.

Combining Theorem~\ref{thm:intro-subthreshold-compactness} with
Proposition~\ref{prop:intro-localized-strict-gap} gives the first existence
theorem.

\begin{introtheorem}[Existence below the spectral bottom]
\label{thm:intro-subthreshold-existence}
Assume \(0<\lambda<\Lambda_0\).  The equation
\[
 P_kU-\lambda U=|U|^{q-2}U
\]
has a nontrivial weak solution
\(U\in\mathcal H_\lambda(M)=H^k(M)\), with the equation understood in
\(\mathcal H_\lambda(M)^*\), provided either
\[
 N\ge4k,
\]
or
\[
 2k+2\le N<4k,
 \qquad
 \Lambda_{\mathrm{loc}}<\lambda<\Lambda_0.
\]
The second interval is nonempty whenever
\[
 \Lambda_0>\Theta_{N,k},
 \qquad
 \Theta_{N,k}
 =
 \frac{2^{4k+1}\Gamma(k+1)^2\Gamma(N/2+2k+1)}
 {(N-2k)\Gamma(N/2)\Gamma(2k+1)}.
\]
This last inequality is a sufficient condition for
\(\Lambda_{\mathrm{loc}}<\Lambda_0\); no necessity is asserted.
\end{introtheorem}

The proof is given in Theorem~\ref{thm:subthreshold-existence}.

At the spectral bottom, \eqref{eq:intro-yang-hsm} and its first-order
counterpart imply, when \(N\ge2k+2\), that there exists
\(c=c(n,m,k)>0\) such that
\[
 Q_{\Lambda_0}(u)
 \ge c(n,m,k)\|u\|_{L^q(M)}^2
 \qquad\bigl(u\in C_c^\infty(M)\bigr).
\]
Thus \(S_{\Lambda_0,k}(M)>0\), and the continuous embeddings
\[
 \mathcal H_{\Lambda_0}(M)\hookrightarrow L^q(M),
 \qquad
 \mathcal E_{\Lambda_0}\hookrightarrow L^q(\R^N)
\]
hold, where \(\mathcal E_{\Lambda_0}\) is the Euclidean threshold form
completion defined in Section~2.  Compactness at the threshold is the content
of the following theorem.

\begin{introtheorem}[Existence at the spectral bottom]
\label{thm:intro-threshold-attainment}
Assume
\[
 N\ge2k+2,
 \qquad
 S_{\Lambda_0,k}(M)<S_{N,k}.
\]
Then \(S_{\Lambda_0,k}(M)\) is attained by a nonzero
\(u\in\mathcal H_{\Lambda_0}(M)\).  If
\(\|u\|_{L^q(M)}=1\), then
\[
 Q_{\Lambda_0}(u)=S_{\Lambda_0,k}(M)
\]
and
\[
 Q_{\Lambda_0}(u,\phi)
 =
 S_{\Lambda_0,k}(M)
 \int_M|u|^{q-2}u\phi\,\dd V
 \qquad
 \bigl(\phi\in\mathcal H_{\Lambda_0}(M)\bigr).
\]
Consequently,
\[
 U=S_{\Lambda_0,k}(M)^{1/(q-2)}u
\]
is a nontrivial weak solution of
\[
 P_kU-\Lambda_0U=|U|^{q-2}U
\]
in \(\mathcal H_{\Lambda_0}(M)^*\).
\end{introtheorem}

This statement is proved in Theorem~\ref{thm:main}.

The strict-gap hypothesis in this threshold theorem holds, in particular,
when \(N\ge4k\).  It also holds when
\[
 2k+2\le N<4k,
 \qquad
 \Lambda_0>\Lambda_{\mathrm{loc}},
\]
by applying the combined strict-gap proposition with
\(\lambda=\Lambda_0\).

At the spectral bottom, the natural space is the completion of
\(C_c^\infty(M)\) in the threshold form norm rather than \(H^k(M)\).  On the
coefficient corresponding to \(\mathscr Y_0=\operatorname{span}\{1\}\),
the form is equivalent to
the homogeneous hyperbolic norm
\[
 \|f\|_{X_k}^2
 =
 \int_0^\infty\!\int_{\Ss^{n-1}}
 \rho^2(1+\rho^2)^{k-1}
 |\widehat f(\rho,\omega)|^2
 |\mathbf c(\rho)|^{-2}\,\dd\omega\,\dd\rho.
\]
Here \(\widehat f\) is the Helgason--Fourier transform,
\(\rho\ge0\) is its spectral parameter,
\(\omega\in\Ss^{n-1}\) is the boundary variable, and \(\mathbf c\) is the
Harish--Chandra \(c\)-function; the normalization is recalled in Section~2.
A low-, intermediate-, and high-frequency decomposition proves cocompactness
of this coefficient function under hyperbolic recentering.  The orthogonal
spherical-harmonic complement is controlled in \(H^k(M)\); bounded-geometry
localization and the refined Sobolev inequality of Meyer--Gérard--Oru
\cite{MeyerGerardOru1997RefinedSobolev} show that any remaining critical
concentration is a Euclidean concentration satisfying
\(\mu_j|\eta_j|\to\infty\).

After such Euclidean escape has been excluded, successive weak limits under
the group preserving \(\mathcal A\) give pairwise orthogonal profiles,
energy and \(L^q\)-mass decoupling, and a remainder that tends to zero in
\(L^q\).  For a normalized threshold minimizing sequence, the strict gap
\[
 S_{\Lambda_0,k}(M)<S_{N,k}
\]
excludes every Euclidean profile escaping \(\mathcal A\), and the
decomposition yields a nonzero weak limit after an axis-preserving recentering.
Hilbert-space splitting, the Brézis--Lieb lemma \cite{BrezisLieb1983}, and
the strict concavity of \(t^{2/q}\) show that the entire critical mass lies in
that limit.  This proves the threshold theorem.

The present paper is organized as follows.  Section~2 records the conformal
model, the variational spaces, and the Helgason--Fourier calculus used
throughout.  Section~3 proves compactness below the spectral bottom and
establishes the localized strict inequalities.  Section~4 proves the
positive critical embedding at \(\lambda=\Lambda_0\).
Section~5 separates the constant spherical eigenspace from its coercive
orthogonal complement and proves the required hyperbolic cocompactness
estimate.  Section~6 constructs the profile decomposition relative to the
axis-preserving group.  Section~7 excludes Euclidean concentration escaping
the singular axis for minimizing sequences satisfying the strict Euclidean
inequality.  Section~8 proves attainment and derives the Euler--Lagrange
equation.  Appendix~\ref{app:scattering} derives the product GJMS multiplier,
the spectral bottom, and the sharp cylindrical Hardy coefficient by
scattering on \(\Hh^{N+1}\).

\paragraph*{Acknowledgments.}
The second author was supported by the National Natural Science Foundation
of China (Grant No.~1257112).  The third author was supported by the Research
Foundation for Youth Scholars of Beijing Technology and Business University
(No.~RFYS2025) and the Research Foundation for Advanced Talents of Beijing
Technology and Business University (No.~19008024091).

\section{Geometric and variational setting}

Write the upper half-space model as
\[
 \Hh^n
 =
 \{(x,r):x\in\R^{n-1},\ r>0\},
 \qquad
 g_{\Hh}=r^{-2}(\dd x^2+\dd r^2).
\]
For
\[
 z=(x,y)\in
 \R^{n-1}\times(\R^{m+1}\setminus\{0\}),
 \qquad
 r=|y|,
 \qquad
 \theta=\frac{y}{|y|},
\]
the product metric satisfies
\[
 g_M=r^{-2}(\dd x^2+\dd y^2).
\]
We write \(\dd V=\dd V_{g_M}\) for its Riemannian volume element.

We use the sign convention \(-\Delta\ge0\).  The letters \(c\) and \(C\)
denote positive constants that may change from line to line; their dependence
is stated whenever it affects a conclusion.  For nonnegative quantities
\(F\) and \(G\), the notation \(F\asymp G\) means that
\[
 cG\le F\le CG
\]
for positive constants \(c,C\) that are uniform in the variables under
consideration and depend only on the fixed parameters indicated in context.
Set
\[
 \gamma=\frac{N-2k}{2},
 \qquad
 a_H=\frac{n-1}{2},
 \qquad
 b_S=\frac{m-1}{2},
 \qquad
 A=\sqrt{-\Delta_{\Hh^n}-a_H^2}.
\]
We recall the spectral representation of \(A\).  Let
\(B=\partial_\infty\Hh^n\simeq\Ss^{n-1}\), and let
\(\langle h,\omega\rangle\) denote the horocycle bracket, normalized so that
the functions
\[
 e_{\rho,\omega}(h)
 =
 \exp\!\bigl((i\rho+a_H)\langle h,\omega\rangle\bigr)
\]
satisfy
\[
 -\Delta_{\Hh^n}e_{\rho,\omega}
 =
 (\rho^2+a_H^2)e_{\rho,\omega}.
\]
For \(f\in C_c^\infty(\Hh^n)\), its Helgason--Fourier transform is
\begin{equation}\label{eq:helgason-fourier-transform}
 \widehat f(\rho,\omega)
 =
 \int_{\Hh^n}
 f(h)\overline{e_{\rho,\omega}(h)}\,\dd V_{\Hh^n}(h),
 \qquad (\rho,\omega)\in\R\times B.
\end{equation}
With normalized surface measure on \(B\), there is a dimensional constant
\(d_n>0\) such that the inversion and Plancherel formulas are
\begin{align}
 f(h)
 &=
 d_n\int_{\R}\!\int_B
 \widehat f(\rho,\omega)e_{\rho,\omega}(h)
 |\mathbf c(\rho)|^{-2}\,\dd\omega\,\dd\rho,
 \label{eq:helgason-inversion}\\
 \|f\|_{L^2(\Hh^n)}^2
 &=
 d_n\int_{\R}\!\int_B
 |\widehat f(\rho,\omega)|^2
 |\mathbf c(\rho)|^{-2}\,\dd\omega\,\dd\rho.
 \label{eq:helgason-plancherel}
\end{align}
Here \(\mathbf c\) is the Harish--Chandra \(c\)-function.  For real
hyperbolic space its Plancherel density can be written, up to a positive
dimensional normalization, as
\begin{equation}\label{eq:harish-chandra-density}
 |\mathbf c(\rho)|^{-2}
 =
 C_n
 \left|
 \frac{\Gamma(a_H+i\rho)}{\Gamma(i\rho)}
 \right|^2.
\end{equation}
Formula \eqref{eq:harish-chandra-density} is read for \(\rho\ne0\), with
the value at \(\rho=0\) defined by continuity.
Consequently,
\begin{equation}\label{eq:c-function-asymptotics}
 |\mathbf c(\rho)|^{-2}\asymp \rho^2
 \quad(0<\rho\le1),
 \qquad
 |\mathbf c(\rho)|^{-2}\asymp \rho^{n-1}
 \quad(\rho\ge1).
\end{equation}
Equations \eqref{eq:helgason-inversion} and
\eqref{eq:helgason-plancherel} are the rank-one specialization of the
Fourier inversion and Plancherel theorems in
\cite[Chapter~III, Section~1]{Helgason1994GeometricAnalysis}; the explicit
real-hyperbolic density in \eqref{eq:harish-chandra-density} follows from the
rank-one Harish--Chandra \(\mathbf c\)-function formula in
\cite[Chapter~IV, Section~6]{Helgason2000GroupsGeometricAnalysis}.  The Weyl
symmetry identifies the
positive and negative spectral halves.  In all subsequent norm formulas we
use the equivalent positive spectral realization, integrate over
\(\rho\ge0\), and absorb the fixed factor into the boundary measure.  In this
normalization, the functional calculus is diagonal:
\begin{equation}\label{eq:helgason-functional-calculus}
 \widehat{Af}(\rho,\omega)=\rho\widehat f(\rho,\omega),
 \qquad
 \widehat{\Phi(A)f}(\rho,\omega)
 =
 \Phi(\rho)\widehat f(\rho,\omega)
\end{equation}
for every bounded Borel function \(\Phi\), initially on the natural spectral
domain and then by closure.  We henceforth use the nonnegative spectral
parameter \(\rho\) for \(A\).

The special-Einstein-product factorization of Case and Malchiodi
\cite[Corollary~3.3]{CaseMalchiodi2024GJMS}, applied with unit curvature
parameter and spherical eigenvalue \(\ell(\ell+m-1)\), yields the multiplier
below.  Appendix~\ref{app:scattering} derives the same formula independently
from the scattering operator of \(\Hh^{N+1}\).
\begin{equation}\label{eq:spherical-symbol}
 p_{k,\ell}(\rho)
 =
 \prod_{j=1}^k
 \left[
  \rho^2+
  \bigl(\ell+b_S+k+1-2j\bigr)^2
 \right].
\end{equation}
Consequently,
\begin{equation}\label{eq:spectral-bottom}
 \Lambda_0
 =
 \inf\sigma_{L^2(M)}(P_k)
 =
 \min_{\ell\ge0}p_{k,\ell}(0).
\end{equation}
This is a continuous spectral bottom, not an \(L^2\)-eigenvalue.

For \(u,v\in C_c^\infty(M)\), define
\[
 Q_\lambda(u,v)
 =
 \int_M(P_ku)v\,\dd V
 -
 \lambda\int_Muv\,\dd V,
 \qquad
 Q_\lambda(u)=Q_\lambda(u,u).
\]
For \(0<\lambda\le\Lambda_0\), set
\[
 \|u\|_\lambda:=Q_\lambda(u)^{1/2},
 \qquad u\in C_c^\infty(M).
\]
The spectral theorem and the fact that \(\Lambda_0\) is not an
\(L^2\)-eigenvalue show that \(\|\cdot\|_\lambda\) is a norm on
\(C_c^\infty(M)\).  We define
\begin{equation}\label{eq:energy-space-completion}
 \mathcal H_\lambda(M)
 :=
 \overline{C_c^\infty(M)}^{\,\|\cdot\|_\lambda},
 \qquad 0<\lambda\le\Lambda_0,
\end{equation}
where the bar denotes closure in the Hilbert completion associated with
\(\|\cdot\|_\lambda\).  For \(0<\lambda<\Lambda_0\), this norm is
equivalent to the \(H^k(M)\)-norm.  At \(\lambda=\Lambda_0\), the
completion \(\Hcal\) does not control the full \(L^2(M)\)-norm.

The general conformal covariance law for the GJMS operator is proved in
\cite{GrahamJenneMasonSparling1992}.  Applying this covariance law to
$g_M=r^{-2}(\dd x^2+\dd y^2)$ and using
$\dd V_{g_M}=r^{-N}\dd z$ shows that the conformal transform
\[
 (\mathscr Cw)(x,r,\theta)
 =
 r^\gamma w(x,r\theta)
\]
satisfies
\begin{align}
 Q_\lambda(\mathscr Cw)
 &=
 \int_{\R^N}(-\Delta)^kw\,w\,\dd z
 -
 \lambda
 \int_{\R^N}\frac{w^2}{|y|^{2k}}\,\dd z,
 \label{eq:conformal-energy}\\
 \|\mathscr Cw\|_{L^q(M)}
 &=
 \|w\|_{L^q(\R^N)}.
 \label{eq:conformal-lq}
\end{align}
Let
\[
 \Ecal
 =
 \overline{
 C_c^\infty(\R^N\setminus\R^{n-1})
 }^{\,\|\cdot\|_{\Ecal}},
\]
where
\[
 \|w\|_{\Ecal}^2
 =
 \|w\|_{\dot H^k(\R^N)}^2
 -
 \Lambda_0
 \int_{\R^N}\frac{w^2}{|y|^{2k}}\,\dd z
\]
and
\[
 \|w\|_{\dot H^k}^2
 =
 \int_{\R^N}(-\Delta)^kw\,w\,\dd z.
\]
Here \(\dot H^k(\R^N)\) is the completion of \(C_c^\infty(\R^N)\) in
this norm.  For an open set \(\Omega\subset\R^N\), we use
\[
 H_0^k(\Omega)
 =\overline{C_c^\infty(\Omega)}^{\,H^k(\Omega)}.
\]
By \eqref{eq:conformal-energy}--\eqref{eq:conformal-lq} and density,
\(\mathscr C\) extends to a unitary isomorphism
\[
 \mathscr C:\Ecal\longrightarrow\Hcal.
\]
Its inverse on the completed spaces is denoted by \(\mathscr C^{-1}\).

For \(0<\lambda\le\Lambda_0\), define
\begin{equation}\label{eq:variational-constant}
 S_{\lambda,k}(M)
 =
 \inf_{u\in C_c^\infty(M)\setminus\{0\}}
 \frac{Q_\lambda(u)}{\|u\|_{L^q(M)}^2}.
\end{equation}
This definition applies both in the coercive range and at the spectral
threshold; in particular, \(S_{\Lambda_0,k}(M)=\Sth\).  The Euclidean
sharp Sobolev constant is
\begin{equation}\label{eq:euclidean-constant}
 S_{N,k}
 =
 \inf_{\varphi\in C_c^\infty(\R^N)\setminus\{0\}}
 \frac{\|\varphi\|_{\dot H^k(\R^N)}^2}
      {\|\varphi\|_{L^q(\R^N)}^2}.
\end{equation}

\section{Extremals below the spectral bottom}

\begin{lemma}[Subthreshold coercivity]
\label{lem:subthreshold-coercivity}
Assume $\Lambda_0>0$ and $0<\lambda<\Lambda_0$.  Then
\begin{equation}
 Q_\lambda(u)
 \ge
 \left(1-\frac{\lambda}{\Lambda_0}\right)
 \int_M(P_ku)u\,\dd V
 \qquad (u\in C_c^\infty(M)).
 \label{eq:subthreshold-coercivity}
\end{equation}
Moreover, $Q_\lambda^{1/2}$ is equivalent to the $H^k(M)$ norm; in
particular, $\Hlam=H^k(M)$ as sets with equivalent Hilbert norms.
\end{lemma}

\begin{proof}
By the definition of the spectral bottom,
\[
 \int_M(P_ku)u\,\dd V
 \ge \Lambda_0\|u\|_2^2.
\]
Hence
\[
 \lambda\|u\|_2^2
 \le
 \frac{\lambda}{\Lambda_0}
 \int_M(P_ku)u\,\dd V,
\]
which proves \eqref{eq:subthreshold-coercivity}.  The multiplier
\eqref{eq:spherical-symbol} also gives constants $c,C>0$, depending only on
$n,m,k$, such that
\[
 c(1+\rho^2+\ell^2)^k
 \le p_{k,\ell}(\rho)
 \le C(1+\rho^2+\ell^2)^k
 \qquad(\rho\ge0,\ \ell\in\mathbb N_0).
\]
The upper bound follows directly from the product formula.  For the lower
bound, $\Lambda_0>0$ excludes a zero factor at $\rho=0$, and the finitely
many remaining low spherical degrees have a positive minimum.  The
Plancherel theorems on \(\Hh^n\) and \(\Ss^m\) therefore show that the
quadratic form of $P_k$ is equivalent to the squared $H^k(M)$-norm.
Combining this with
\eqref{eq:subthreshold-coercivity} proves the last assertion.
\end{proof}

\begin{theorem}[Subthreshold compactness criterion]
\label{thm:subthreshold-compactness}
Assume
\[
 \Lambda_0>0,
 \qquad
 0<\lambda<\Lambda_0,
 \qquad
 \Slam<S_{N,k}.
\]
Then $\Slam$ is attained by a nonzero $u\in\Hlam$.  After normalizing
$\|u\|_q=1$,
\[
 Q_\lambda(u)=\Slam
\]
and
\begin{equation}
 Q_\lambda(u,\phi)
 =
 \Slam\int_M|u|^{q-2}u\phi\,\dd V
 \qquad(\phi\in\Hlam).
 \label{eq:subthreshold-EL}
\end{equation}
Consequently, $U=\Slam^{1/(q-2)}u$ is a nontrivial weak solution of
\[
 P_kU-\lambda U=|U|^{q-2}U.
\]
\end{theorem}

\begin{proof}
Let $u_j\in C_c^\infty(M)$ be a minimizing sequence normalized by
$\|u_j\|_q=1$.  Lemma~\ref{lem:subthreshold-coercivity} and the Sobolev
embedding give constants $c_\lambda,c_\lambda'>0$ such that
\[
 Q_\lambda(v)
 \ge c_\lambda\|v\|_{H^k(M)}^2
 \ge c_\lambda'\|v\|_q^2
 \qquad(v\in H^k(M)).
\]
Thus $0<\Slam<\infty$, and $(u_j)$ is bounded in $H^k(M)$.

We shall use the following localization property of $Q_\lambda$.  Since
$P_k$ has
order $2k$ and principal part $(-\Delta_g)^k$, integration by parts in
bounded-geometry frames gives
\[
 Q_\lambda(v)
 =
 \sum_{r,s=0}^k
 \int_M
 \langle \mathsf a_{rs}\nabla^rv,\nabla^sv\rangle\,\dd V,
\]
where the coefficient fields are uniformly bounded.  Suppose that
$\zeta_{1,j},\ldots,\zeta_{L,j}$ are smooth functions satisfying
$\sum_{a=1}^L\zeta_{a,j}^2=1$ and
\[
 \max_{\substack{1\le a\le L\\1\le r\le k}}
 \|\nabla^r\zeta_{a,j}\|_\infty\longrightarrow0.
\]
For \(0\le r\le k\), the Leibniz rule has the schematic form
\begin{equation}\label{eq:leibniz-localization}
 \nabla^r(\zeta_av)
 =
 \zeta_a\nabla^rv
 +\sum_{t=1}^r C_{r,t}
   (\nabla^t\zeta_a)*\nabla^{r-t}v.
\end{equation}
After this identity is substituted into the local expression for
\(Q_\lambda\), the terms on which no derivative falls on a cutoff sum to
\(Q_\lambda(v)\), because \(\sum_a\zeta_a^2=1\).  Every remaining
integrand contains at least one positive-order derivative of a cutoff and is
a product of two derivatives of \(v\), one of order at most \(k\) and the
other of order at most \(k-1\).  Cauchy--Schwarz therefore gives
\begin{equation}\label{eq:localization-error-bound}
 \left|Q_\lambda(v)-\sum_{a=1}^LQ_\lambda(\zeta_av)\right|
 \le
 C\max_{\substack{1\le a\le L\\1\le t\le k}}
 \|\nabla^t\zeta_a\|_\infty
 \|v\|_{H^k(M)}^2
\end{equation}
when the cutoff derivatives are uniformly small.  Hence, for every bounded
sequence $(v_j)$ in $H^k(M)$,
\begin{equation}
 Q_\lambda(v_j)
 =
 \sum_{a=1}^LQ_\lambda(\zeta_{a,j}v_j)+o(1).
 \label{eq:subthreshold-slow-localization}
\end{equation}
For a fixed quadratic partition whose transition region lies in a compact
set \(K\), every error term in \eqref{eq:leibniz-localization} is supported in
\(K\).  Applying Cauchy--Schwarz with the factor of derivative order at most
\(k-1\) in \(H^{k-1}(K)\) gives
\begin{equation}
 \left|
 Q_\lambda(v)-\sum_{a=1}^LQ_\lambda(\zeta_av)
 \right|
 \le
 C_{\zeta_1,\ldots,\zeta_L}
 \|v\|_{H^k(K)}\|v\|_{H^{k-1}(K)}.
 \label{eq:subthreshold-fixed-localization}
\end{equation}

Let $\mathfrak m_j=|u_j|^q\,\dd V$.  We record the bounded-geometry estimate
used to rule out vanishing.  Choose \(r>0\) smaller than the uniform normal
coordinate radius and a cover \(B_i=B(x_i,r)\) such that the doubled balls
\(B_i^*=B(x_i,2r)\) have overlap at most \(L=L(n,m)\).  Uniform equivalence
of the metric coefficients in these charts and the Euclidean critical
Sobolev inequality give
\[
 \left(\int_{B_i}|v|^q\,\dd V\right)^{2/q}
 \le C\|v\|_{H^k(B_i^*)}^2,
\]
with \(C\) independent of \(i\).  If
\[
 M_r(v):=\sup_{P\in M}\int_{B(P,r)}|v|^q\,\dd V,
\]
then
\begin{align*}
 \int_M|v|^q\,\dd V
 &\le
 \sum_i
 \left(\int_{B_i}|v|^q\,\dd V\right)^{1-2/q}
 \left(\int_{B_i}|v|^q\,\dd V\right)^{2/q}\\
 &\le
 C M_r(v)^{1-2/q}\sum_i\|v\|_{H^k(B_i^*)}^2\\
 &\le
 CL M_r(v)^{1-2/q}\|v\|_{H^k(M)}^2.
\end{align*}
Thus the uniform local Sobolev estimate is
\begin{equation}\label{eq:uniform-local-sobolev}
 \int_M|v|^q\,\dd V
 \le
 C\left(\sup_{P\in M}\int_{B(P,r)}|v|^q\,\dd V\right)^{1-2/q}
 \|v\|_{H^k(M)}^2.
\end{equation}
Since \(\|u_j\|_q=1\) and \((u_j)\) is bounded in \(H^k(M)\),
\eqref{eq:uniform-local-sobolev} excludes vanishing.  The
concentration--compactness lemma of Lions
\cite[Lemma~I.1]{Lions1984ConcentrationCompactness} applies to the probability
measures $\mathfrak m_j$ with geodesic balls in place of Euclidean balls.
Properness ensures that these balls are compact, and bounded geometry
provides smooth separating cutoffs with uniform derivative bounds.  If
dichotomy occurred with masses $a$ and $1-a$, where
$0<a<1$, its separating annuli and bounded geometry would give smooth
functions $\chi_j,\psi_j$ such that
\[
 \chi_j^2+\psi_j^2=1,
 \qquad
 \int_M|\chi_ju_j|^q\,\dd V\longrightarrow a,
 \qquad
 \int_M|\psi_ju_j|^q\,\dd V\longrightarrow1-a,
\]
and the derivatives of positive order of both cutoffs would tend uniformly
to zero.  Such a partition is obtained by applying sine and cosine to a
uniformly smoothed distance function across an annulus whose width tends to
infinity.  Equation~\eqref{eq:subthreshold-slow-localization} and the
variational inequality $Q_\lambda(v)\ge\Slam\|v\|_q^2$ would then imply
\[
 \Slam
 \ge
 \Slam\bigl(a^{2/q}+(1-a)^{2/q}\bigr),
\]
contrary to $\Slam>0$ and $2/q<1$.

The compactness alternative therefore holds.  Hence there are points
$x_j\in M$ such that, for every $\eta>0$, some $R_\eta>0$ satisfies
\[
 \int_{B(x_j,R_\eta)}|u_j|^q\,\dd V\ge1-\eta
\]
for all sufficiently large $j$.  Product isometries move $x_j$ to a fixed
point without changing either $Q_\lambda$ or the $L^q$ norm.  After this
recentering and passage to a subsequence,
\[
 u_j\rightharpoonup u\quad\text{in }H^k(M),
 \qquad
 u_j\longrightarrow u\quad\text{in }H^{k-1}_{\mathrm{loc}}(M)
\]
and almost everywhere.  The measures $|u_j|^q\,\dd V$ are tight.

Suppose that $u=0$.  We use the following specialization of Lions' second
concentration--compactness lemma.  In Euclidean space, put
\[
 \mathscr D_kv
 =
 \begin{cases}
  (-\Delta)^{k/2}v, & k \text{ even},\\
  \nabla(-\Delta)^{(k-1)/2}v, & k \text{ odd}.
 \end{cases}
\]
Thus
\(\|\mathscr D_kv\|_2^2=\int_{\R^N}(-\Delta)^kv\,v\,\dd z\).
With derivative order \(k\), \(p=2\), and \(q=2N/(N-2k)\), Lions' lemma
\cite[Lemma~I.1 and Section~I.4(i)]{Lions1985LimitCase} is as follows.  If
\(v_j\) is bounded in \(\dot H^k(\R^N)\),
\(v_j\rightharpoonup v\) weakly, and
\[
 |\mathscr D_kv_j|^2\,\dd z\rightharpoonup\mu,
 \qquad
 |v_j|^q\,\dd z\rightharpoonup\nu
\]
as bounded nonnegative measures, with the second convergence tight, then
there is an at most countable index set \(J\), together with pairwise
distinct points \(z_i\in\R^N\) and numbers \(\nu_i,\mu_i>0\), such that
\begin{align}
 \nu
 &=
 |v|^q\,\dd z+\sum_{i\in J}\nu_i\delta_{z_i},
 \label{eq:lions-measure-mass}\\
 \mu
 &\ge
 |\mathscr D_kv|^2\,\dd z+\sum_{i\in J}\mu_i\delta_{z_i},
 \qquad
 S_{N,k}\nu_i^{2/q}\le\mu_i.
 \label{eq:lions-measure-energy}
\end{align}
The sharp constant in Lions' statement is \(S_{N,k}\) in our normalization.
The proof begins by applying the sharp Sobolev inequality to
a compactly supported multiplier of \(v_j\); when \(v=0\), this yields
\begin{equation}\label{eq:lions-measure-test-inequality}
 S_{N,k}
 \left(\int_{\R^N}|\varphi|^q\,\dd\nu\right)^{2/q}
 \le
 \int_{\R^N}\varphi^2\,\dd\mu
 \qquad
 \bigl(\varphi\in C_c^\infty(\R^N)\bigr),
\end{equation}
from which the atomic conclusion and
\eqref{eq:lions-measure-energy} follow.

We verify the hypotheses after localization on \(M\).  Define the
Riemannian principal differential expression
\[
 \mathscr D_{g,k}v
 =
 \begin{cases}
  (-\Delta_g)^{k/2}v, & k \text{ even},\\
  \nabla_g(-\Delta_g)^{(k-1)/2}v, & k \text{ odd}.
 \end{cases}
\]
After passage to a subsequence, the bounded nonnegative measures satisfy
\[
 |u_j|^q\,\dd V\rightharpoonup\nu,
 \qquad
 |\mathscr D_{g,k}u_j|^2\,\dd V\rightharpoonup\mathfrak e.
\]
The measure \(\mathfrak e\) is finite because \((u_j)\) is bounded in
\(H^k(M)\).  Fix \(\varphi\in C_c^\infty(M)\), and choose a relatively
compact open set \(\Omega_\varphi\) containing
\(\operatorname{supp}\varphi\).  For every \(\varepsilon>0\), cover
\(\overline{\Omega_\varphi}\) by finitely many normal-coordinate balls whose
radii are small enough that the principal metric coefficients differ from
their Euclidean values by at most the prescribed error.  Choose a smooth
quadratic partition \(\sum_a\zeta_a^2=1\) on a neighborhood of
\(\operatorname{supp}\varphi\).  Since \(q>2\),
\[
 \|w\|_q^2
 =
 \bigl\||w|^2\bigr\|_{q/2}
 \le
 \sum_a\bigl\||\zeta_aw|^2\bigr\|_{q/2}
 =
 \sum_a\|\zeta_aw\|_q^2.
\]
Apply the sharp Euclidean Sobolev inequality to
\((\zeta_a\varphi u_j)\) in each coordinate ball and sum in \(a\).  In the
expansion of \(\mathscr D_{g,k}(\zeta_a\varphi u_j)\), every term other than
\(\zeta_a\varphi\mathscr D_{g,k}u_j\) contains at most \(k-1\) derivatives of
\(u_j\).  The quadratic partition combines the principal terms, while
Cauchy--Schwarz and Young's inequality absorb the cross terms into the
principal energy.  After decreasing the coordinate radii, this gives
\[
 (S_{N,k}-\varepsilon)\|\varphi u_j\|_q^2
 \le
 \int_M\varphi^2|\mathscr D_{g,k}u_j|^2\,\dd V
 +C_{\varepsilon,\varphi}
  \|u_j\|_{H^{k-1}(\Omega_\varphi)}^2.
\]
For every summand, \(|\zeta_a\varphi u_j|^q\,\dd V\) has support in a fixed
compact set, so the nonlinear measure convergence is tight on that chart,
as required in Lions' lemma.
The last term tends to zero because
\(u_j\to0\) in \(H^{k-1}_{\mathrm{loc}}(M)\).  Passing first to the
measure limits and then letting \(\varepsilon\downarrow0\) gives
\begin{equation}
 S_{N,k}
 \left(\int_M|\varphi|^q\,\dd\nu\right)^{2/q}
 \le
 \int_M\varphi^2\,\dd\mathfrak e
 \qquad(\varphi\in C_c^\infty(M)).
 \label{eq:subthreshold-defect-measure}
\end{equation}

Inequality \eqref{eq:subthreshold-defect-measure} forces \(\nu\) to be
purely atomic.  Indeed, suppose that its nonatomic part has mass \(a>0\).
For any \(L\ge1\), divide this part into
pairwise disjoint Borel sets of \(\nu\)-mass \(a/L\).  By inner regularity,
choose pairwise disjoint compact subsets \(K_1,\ldots,K_L\) carrying at least
\(a/(2L)\) mass each.  Since this family is finite, the compact sets have
pairwise disjoint neighborhoods; choose
\(\varphi_r\in C_c^\infty(M)\), \(0\le\varphi_r\le1\), equal to one on
\(K_r\), with pairwise disjoint supports.  Applying
\eqref{eq:subthreshold-defect-measure} to each \(\varphi_r\) and summing gives
\[
 \mathfrak e(M)
 \ge
 S_{N,k}\sum_{r=1}^L\nu(K_r)^{2/q}
 \ge
 S_{N,k}\left(\frac a2\right)^{2/q}L^{1-2/q}.
\]
This contradicts \(\mathfrak e(M)<\infty\), because \(1-2/q>0\).  Hence
\(\nu\) is purely atomic and has at most countably many atoms.  Write them as
\(\{x_i:i\in I\}\), with masses \(\nu_i=\nu(\{x_i\})>0\).

For a fixed atom \(x_i\), choose cutoffs
\(0\le\psi_r\le1\) that equal one on \(B(x_i,r)\) and are supported in
\(B(x_i,2r)\).  Equation \eqref{eq:subthreshold-defect-measure} and continuity
from above for the finite measures \(\nu\) and \(\mathfrak e\) give, as
\(r\downarrow0\),
\[
 \mathfrak e(\{x_i\})
 \ge
 S_{N,k}\nu_i^{2/q}.
\]
Finally, the tight convergence of the probability measures
\(|u_j|^q\,\dd V\) implies \(\nu(M)=1\).  We have therefore proved
\begin{equation}
 \nu=\sum_{i\in I}\nu_i\delta_{x_i},
 \qquad
 \nu_i>0,
 \qquad
 \sum_{i\in I}\nu_i=1,
 \qquad
 \mathfrak e(\{x_i\})\ge S_{N,k}\nu_i^{2/q}.
 \label{eq:subthreshold-atomic-mass}
\end{equation}

Fix a finite set $F\subset I$.  For sufficiently small $\delta>0$, choose
disjoint balls $B(x_i,2\delta)$, $i\in F$, and functions
\(\vartheta_i^{(\delta)}\in C_c^\infty(B(x_i,2\delta))\) equal to one on
\(B(x_i,\delta)\).  Put
\[
 \chi_i^{(\delta)}
 =
 \sin\!\left(\frac\pi2\vartheta_i^{(\delta)}\right),
 \qquad
 \chi_0^{(\delta)}
 =
 \prod_{i\in F}
 \cos\!\left(\frac\pi2\vartheta_i^{(\delta)}\right).
\]
The supports of the \(\vartheta_i^{(\delta)}\) are disjoint, so at each
point at most one factor differs from one.  Hence
\(\chi_0^{(\delta)}\) represents the whole complement of the selected
neighborhoods, including the atoms indexed by \(I\setminus F\), and
\[
 (\chi_0^{(\delta)})^2
 +\sum_{i\in F}(\chi_i^{(\delta)})^2=1,
\]
where $\chi_i^{(\delta)}=1$ on $B(x_i,\delta)$.  For fixed $\delta$,
Equation~\eqref{eq:subthreshold-fixed-localization} and local $H^{k-1}$
convergence give
\[
 Q_\lambda(u_j)
 =
 Q_\lambda(\chi_0^{(\delta)}u_j)
 +\sum_{i\in F}Q_\lambda(\chi_i^{(\delta)}u_j)+o(1).
\]
There is no sum over \(i\notin F\): all of those regions are contained in the
single term \(Q_\lambda(\chi_0^{(\delta)}u_j)\).  This term is nonnegative by
subthreshold coercivity.  For \(i\in F\), local \(H^{k-1}\)-convergence makes
all lower-order terms and cutoff commutators vanish, whereas the
principal-part energy measure satisfies
\[
 \liminf_{j\to\infty}
 Q_\lambda(\chi_i^{(\delta)}u_j)
 \ge
 \int_M(\chi_i^{(\delta)})^2\,\dd\mathfrak e
 \ge
 \mathfrak e(\{x_i\})
 \ge
 S_{N,k}\nu_i^{2/q}.
\]
Summing over the finite set \(F\), then letting \(j\to\infty\) and
\(\delta\downarrow0\), gives
\[
 \Slam\ge S_{N,k}\sum_{i\in F}\nu_i^{2/q}.
\]
Finally let $F$ increase to $I$.  By
\eqref{eq:subthreshold-atomic-mass},
\[
 \Slam
 \ge S_{N,k}\sum_{i\in I}\nu_i^{2/q}
 \ge S_{N,k}\left(\sum_{i\in I}\nu_i\right)^{2/q}
 =S_{N,k},
\]
contrary to $\Slam<S_{N,k}$.  Therefore $u\ne0$.

Put $v_j=u_j-u$.  Local compactness gives almost-everywhere convergence after
a subsequence, so the Hilbert identity and the Brézis--Lieb lemma
\cite[Theorem~1]{BrezisLieb1983} give
\begin{align*}
 Q_\lambda(u_j)
 &=Q_\lambda(u)+Q_\lambda(v_j)+o(1),\\
 1
 &=\|u\|_q^q+\|v_j\|_q^q+o(1).
\end{align*}
Set $a=\|u\|_q^q>0$ and, after a further subsequence,
$b=\lim_j\|v_j\|_q^q$.  Then $a+b=1$ and
\[
 \Slam
 \ge
 \Slam\bigl(a^{2/q}+b^{2/q}\bigr).
\]
Because $2/q<1$, this forces $b=0$ and $a=1$.  The energy splitting and the
variational lower bound then give $Q_\lambda(u)=\Slam$.  Differentiating the
quotient proves \eqref{eq:subthreshold-EL}, and the stated scaling gives the
equation with unit nonlinear coefficient.
\end{proof}

\subsection{Strict inequalities}

\begin{lemma}[High-dimensional localized bubble]
\label{lem:strict-high-dim}
Assume $N\ge4k$ and $0<\lambda\le\Lambda_0$.  Then
\[
 \inf_{u\in C_c^\infty(M)\setminus\{0\}}
 \frac{Q_\lambda(u)}{\|u\|_q^2}
 <S_{N,k}.
\]
\end{lemma}

\begin{proof}
Choose $p_0=(0,y_0)$ with $|y_0|=1$ in the conformal Euclidean model.  Let
$\chi\in C_c^\infty(B_{3/4}(p_0))$ satisfy $\chi=1$ on $B_{1/2}(p_0)$, and
set
\[
 w_\varepsilon(z)
 =
 \chi(z)
 \left(
  \frac{2\sqrt\varepsilon}
       {\varepsilon+|z-p_0|^2}
 \right)^\gamma.
\]
Put \(\delta=\sqrt\varepsilon\) and denote the uncut function by
\[
 U_\delta(z)
 =
 \left(\frac{2\delta}{\delta^2+|z-p_0|^2}\right)^\gamma,
 \qquad w_\varepsilon=\chi U_\delta.
\]
Thus \(\delta\) is its concentration scale.  The estimates below are the
corresponding localized-bubble expansions in the present normalization;
compare the second-order computation of Brézis--Nirenberg
\cite[Lemma~1.1 and Eqs.~(1.10)--(1.13)]{BrezisNirenberg1983} and the
polyharmonic computation in
\cite[Lemmas~5.1--5.2]{LuYang2022GreenGJMS}.

Let
\[
 \mathcal R=\left\{z:\frac12\le |z-p_0|\le\frac34\right\}.
\]
For every multi-index \(\alpha\) with \(|\alpha|\le2k\), direct
differentiation gives
\[
 |\partial^\alpha U_\delta(z)|
 \le C_\alpha\delta^\gamma
 \qquad (z\in\mathcal R).
\]
Moreover, for \(|z-p_0|\ge1/2\),
\[
 |\partial^\alpha U_\delta(z)|
 \le
 C_\alpha\delta^\gamma
 |z-p_0|^{-2\gamma-|\alpha|}.
\]
After integrating by parts in the \(\dot H^k\)-energy, expand
\[
 (-\Delta)^k(\chi U_\delta)
 =
 \chi(-\Delta)^kU_\delta
 +[(-\Delta)^k,\chi]U_\delta.
\]
Every term in the commutator is supported in \(\mathcal R\) and contains a
positive-order derivative of \(\chi\).  The two derivative bounds above
therefore imply
\begin{equation}\label{eq:localized-bubble-energy-error}
 \left|
 \|w_\varepsilon\|_{\dot H^k}^2
 -\|U_\delta\|_{\dot H^k}^2
 \right|
 \le C\delta^{2\gamma}
 =C\varepsilon^\gamma.
\end{equation}
Since \(\gamma q=N\),
\[
 \int_{\{|z-p_0|\ge1/2\}}U_\delta^q\,\dd z
 =O(\delta^N),
\]
and hence
\[
 \left|
 \|w_\varepsilon\|_q^2-\|U_\delta\|_q^2
 \right|
 =O(\delta^N).
\]
The uncut function is a Euclidean extremal, so
\(\|U_\delta\|_{\dot H^k}^2=S_{N,k}\|U_\delta\|_q^2\).  Consequently,
\begin{equation}\label{eq:localized-bubble-sobolev-expansion}
 \|w_\varepsilon\|_{\dot H^k}^2
 =
 S_{N,k}\|w_\varepsilon\|_q^2
 +O(\varepsilon^\gamma).
\end{equation}

For the weighted term, the change of variables \(z=p_0+\delta Z\) gives
\begin{equation}\label{eq:localized-bubble-weight-scaling}
\begin{aligned}
 I_\delta
 &:=
 \int_{\R^N}\frac{w_\varepsilon^2}{|y|^{2k}}\,\dd z\\
 &=
 2^{2\gamma}\delta^{2k}
 \int_{\R^N}
 \frac{\chi(p_0+\delta Z)^2}
 {|y_0+\delta Z_y|^{2k}}
 (1+|Z|^2)^{-2\gamma}\,\dd Z.
\end{aligned}
\end{equation}
If \(N>4k\), then
\((1+|Z|^2)^{-2\gamma}\in L^1(\R^N)\).  Dominated convergence in
\eqref{eq:localized-bubble-weight-scaling} yields
\[
 I_\delta
 =
 c_0\delta^{2k}+o(\delta^{2k})
 =
 c_0\varepsilon^k+o(\varepsilon^k),
 \qquad
 c_0
 =
 2^{2\gamma}\int_{\R^N}(1+|Z|^2)^{-2\gamma}\,\dd Z>0.
\]
If \(N=4k\), the radial integrand has order \(r^{-1}\) at infinity, and
integration over \(1\le r\le c/\delta\) gives
\[
 I_\delta
 =
 c_1\delta^{2k}\log\frac1\delta+O(\delta^{2k})
 =
 c_2\varepsilon^k|\log\varepsilon|+O(\varepsilon^k),
 \qquad c_1,c_2>0.
\]
When \(N>4k\), one has \(\gamma=(N-2k)/2>k\); hence the negative term
\(-\lambda c_0\varepsilon^k\) strictly dominates the error
\(O(\varepsilon^\gamma)\) in
\eqref{eq:localized-bubble-sobolev-expansion}.  When \(N=4k\), the logarithm
provides the strict domination.  Equations
\eqref{eq:conformal-energy}--\eqref{eq:conformal-lq} therefore give a quotient
strictly below \(S_{N,k}\) for all sufficiently small \(\varepsilon\).
\end{proof}

For the lower-dimensional range, recall the explicit positive quantities
\(\mathcal A_{N,k}^{\mathrm{loc}}\), \(\mathcal F_{n,m,k}\), and
\(\Lambda_{\mathrm{loc}}\) from
\eqref{eq:A-local}--\eqref{eq:lambda-local}.

\begin{proposition}[Lower-dimensional localized bubble]
\label{prop:strict-low-dim}
Assume $2k+2\le N<4k$ and $0<\lambda\le\Lambda_0$.  If
\[
 \lambda>\Lambda_{\mathrm{loc}},
\]
then
\[
 \inf_{u\in C_c^\infty(M)\setminus\{0\}}
 \frac{Q_\lambda(u)}{\|u\|_q^2}
 <S_{N,k}.
\]
\end{proposition}

\begin{proof}
Fix $p_0=(0,y_0)$ with $|y_0|=1$, and put
$\varrho=|z-p_0|$.  In $B_1(p_0)$ set
\[
 \phi_\varepsilon(z)
 =
 \left(\frac{2\sqrt\varepsilon}{\varepsilon+\varrho^2}\right)^\gamma,
 \qquad
 X_\varepsilon(\varrho)
 =
 \frac{1-\varrho^2}{1+\varepsilon}.
\]
For $a\in\R$, write $(a)_0=1$ and
$(a)_j=a(a+1)\cdots(a+j-1)$ for $j\ge1$.  Following the
boundary-corrected ball family of Lu and Yang
\cite[Eqs.~(5.12) and (5.13)]
{LuYang2022GreenGJMS}, define
\[
 H_\varepsilon
 =
 \left(\frac{2\sqrt\varepsilon}{1+\varepsilon}\right)^\gamma
 \sum_{j=0}^{k-1}\frac{(\gamma)_j}{j!}X_\varepsilon^j,
 \qquad
 \widetilde f_\varepsilon=\phi_\varepsilon-H_\varepsilon.
\]
The function $\widetilde f_\varepsilon$ vanishes with its first $k-1$ normal
derivatives on $\partial B_1(p_0)$.  Moreover, $H_\varepsilon$ is a polynomial
of degree at most $k-1$ in $\varrho^2$, so
$(-\Delta)^kH_\varepsilon=0$.  Set
\[
 C_{N,k}
 =
 2^\gamma
 \frac{\Gamma(N/2)}{\Gamma(k+1)\Gamma(\gamma)},
 \qquad
 f_\varepsilon=C_{N,k}^{-1}\widetilde f_\varepsilon,
\]
and extend $f_\varepsilon$ by zero outside $B_1(p_0)$.  The extension belongs
to $H_0^k(B_1(p_0))$.

We compute the coefficient in the present normalization.  Put
\[
 L_{N,k}=\frac{\Gamma(N/2+k)}{\Gamma(N/2-k)},
 \qquad
 I_0=|\Ss^N|.
\]
Direct differentiation and radial integration give
\[
 (-\Delta)^k\phi_\varepsilon
 =L_{N,k}\phi_\varepsilon^{q-1},
 \qquad
 \int_{\R^N}\phi_\varepsilon^q\,\dd z=I_0,
 \qquad
 S_{N,k}=L_{N,k}I_0^{1-2/q}.
\]
Let
\[
 J_\varepsilon
 =\int_{B_1(p_0)}
   \phi_\varepsilon^{q-1}H_\varepsilon\,\dd z.
\]
Since $(-\Delta)^kH_\varepsilon=0$ and
$\widetilde f_\varepsilon\in H_0^k(B_1(p_0))$, integration by parts gives
\begin{equation}
 \int_{\R^N}(-\Delta)^kf_\varepsilon f_\varepsilon\,\dd z
 =C_{N,k}^{-2}L_{N,k}(I_0-J_\varepsilon)
  +o(\varepsilon^\gamma).
 \label{eq:local-energy-part}
\end{equation}
Here the omitted bubble tail is $O(\varepsilon^{N/2})$.
The Taylor series for $(1-X)^{-\gamma}$ shows that
$0\le H_\varepsilon\le\phi_\varepsilon$ in the ball.  Direct scaling gives
\[
 \int_{B_1(p_0)}
 \phi_\varepsilon^{q-2}H_\varepsilon^2\,\dd z
 =O(\varepsilon^{2\gamma}),
 \qquad
 \int_{B_1(p_0)}H_\varepsilon^q\,\dd z
 =O(\varepsilon^{N/2}).
\]
Consequently,
\[
 \int_{\R^N}|f_\varepsilon|^q\,\dd z
 =C_{N,k}^{-q}
  (I_0-qJ_\varepsilon+o(\varepsilon^\gamma)),
\]
and hence
\begin{equation}
 \int_{\R^N}(-\Delta)^kf_\varepsilon f_\varepsilon\,\dd z
 -S_{N,k}\|f_\varepsilon\|_q^2
 =C_{N,k}^{-2}L_{N,k}J_\varepsilon
  +o(\varepsilon^\gamma).
 \label{eq:local-energy-reduction}
\end{equation}

Set $R_{k-1}(X)=\sum_{j=0}^{k-1}(\gamma)_jX^j/j!$.  With
$\delta=\sqrt\varepsilon$, the change of variables
$\varrho=\delta t$ gives
\begin{align*}
 \lim_{\varepsilon\downarrow0}
 \varepsilon^{-\gamma}J_\varepsilon
 &=2^N|\Ss^{N-1}|R_{k-1}(1)
   \int_0^\infty
   \frac{t^{N-1}}{(1+t^2)^{N/2+k}}\,\dd t,\\
 R_{k-1}(1)
 &=\frac{\Gamma(N/2)}{\Gamma(\gamma+1)\Gamma(k)},\\
 \int_0^\infty
 \frac{t^{N-1}}{(1+t^2)^{N/2+k}}\,\dd t
 &=\frac{\Gamma(N/2)\Gamma(k)}{2\Gamma(N/2+k)}.
\end{align*}
Substitution of $C_{N,k}$ and $L_{N,k}$ into
\eqref{eq:local-energy-reduction} therefore yields
\begin{equation}
 \int_{\R^N}(-\Delta)^kf_\varepsilon f_\varepsilon\,\dd z
 -S_{N,k}\|f_\varepsilon\|_q^2
 =
 \mathcal A_{N,k}^{\mathrm{loc}}\varepsilon^\gamma
 +o(\varepsilon^\gamma).
 \label{eq:local-energy-expansion}
\end{equation}

It remains to evaluate the weighted lower-order term specific to the product
geometry.  The hypergeometric
remainder identity \cite[Eq.~(5.20)]{LuYang2022GreenGJMS} gives, for
$0<\varrho<1$,
\begin{equation}
 \varepsilon^{-\gamma/2}f_\varepsilon
 =
 (1+\varepsilon)^{-\gamma}X_\varepsilon^k
 {}_2F_1\!\left(\frac N2,1;k+1;X_\varepsilon\right).
 \label{eq:corrected-bubble-identity}
\end{equation}
Consequently, away from $p_0$,
\[
 \varepsilon^{-\gamma/2}f_\varepsilon
 \longrightarrow
 (1-\varrho^2)^k
 {}_2F_1\!\left(\frac N2,1;k+1;1-\varrho^2\right).
\]
Euler's integral representation gives the uniform majorant needed below.
For \(0\le X<1\),
\[
 {}_2F_1\!\left(\frac N2,1;k+1;X\right)
 =
 k\int_0^1
 (1-t)^{k-1}(1-Xt)^{-N/2}\,\dd t.
\]
After the substitution \(s=1-t\), the right-hand side is
\[
 k\int_0^1
 s^{k-1}\bigl((1-X)+Xs\bigr)^{-N/2}\,\dd s.
\]
It is uniformly bounded for \(0\le X\le1/2\).  If \(1/2<X<1\), then,
with \(u=s/(1-X)\),
\begin{align*}
 {}_2F_1\!\left(\frac N2,1;k+1;X\right)
 &\le
 C\int_0^1
 s^{k-1}\bigl((1-X)+s\bigr)^{-N/2}\,\dd s\\
 &\le
 C(1-X)^{k-N/2}
 \int_0^\infty u^{k-1}(1+u)^{-N/2}\,\dd u\\
 &\le C(1-X)^{-\gamma}.
\end{align*}
The last integral is finite because \(k<N/2\).  Since
\[
 1-X_\varepsilon(\varrho)
 =
 \frac{\varepsilon+\varrho^2}{1+\varepsilon},
 \qquad
 X_\varepsilon(\varrho)^k
 =
 \frac{(1-\varrho^2)^k}{(1+\varepsilon)^k},
\]
\eqref{eq:corrected-bubble-identity} now gives
\begin{align*}
 \varepsilon^{-\gamma/2}|f_\varepsilon(z)|
 &\le
 C(1-\varrho^2)^k(\varepsilon+\varrho^2)^{-\gamma}\\
 &\le
 C\varrho^{-2\gamma}(1-\varrho^2)^k.
\end{align*}
Thus
\begin{equation}
 \varepsilon^{-\gamma/2}|f_\varepsilon(z)|
 \le
 C\varrho^{-(N-2k)}(1-\varrho^2)^k,
 \qquad 0<\varrho<1.
 \label{eq:corrected-bubble-majorant}
\end{equation}
Indeed, the square has radial order
\(\varrho^{N-1-2(N-2k)}\) near \(p_0\), which is integrable precisely when
\(N<4k\).  At the point where $\partial B_1(p_0)$ meets the singular axis,
\[
 1-\varrho^2
 =2y\cdot y_0-|x|^2-|y|^2
 \le2|y|.
\]
Thus
\[
 |y|^{-2k}(1-\varrho^2)^{2k}\le2^{2k},
\]
so the square of the majorant remains integrable after multiplication by
$|y|^{-2k}$; away from this boundary point the weight is bounded.

Write $z-p_0=(\xi,\zeta)$, pass to $s=|\xi|$ and $t=|\zeta|$, and apply the
Funk--Hecke formula to $|y_0+\zeta|^{-2k}$.  The resulting angular integral
is the expression in \eqref{eq:F-local}.  Dominated convergence, justified by
\eqref{eq:corrected-bubble-majorant}, gives
\begin{equation}
 \int_{\R^N}\frac{f_\varepsilon^2}{|y|^{2k}}\,\dd z
 =
 \mathcal F_{n,m,k}\varepsilon^\gamma
 +o(\varepsilon^\gamma).
 \label{eq:local-weight-expansion}
\end{equation}

To obtain admissible smooth test functions, put
$d(z)=1-\varrho=\operatorname{dist}(z,\partial B_1(p_0))$.  Since
$d\le1-\varrho^2\le2|y|$, the integer-order Hardy inequality in boundary
normal coordinates gives
\begin{equation}
 \int_{B_1(p_0)}\frac{|h|^2}{|y|^{2k}}\,\dd z
 \le
 C\int_{B_1(p_0)}\frac{|h|^2}{d^{2k}}\,\dd z
 \le C\|h\|_{H^k(B_1(p_0))}^2
 \qquad(h\in H_0^k(B_1(p_0))).
 \label{eq:tangent-ball-hardy}
\end{equation}
For each $\varepsilon$, density of $C_c^\infty(B_1(p_0))$ in
$H_0^k(B_1(p_0))$ provides $f_\varepsilon^\sharp$ such that
\[
 \|f_\varepsilon^\sharp-f_\varepsilon\|_{H^k(B_1(p_0))}
 \le\varepsilon^{\gamma+1}.
\]
The collar width and mollification scale may be chosen diagonally to tend to
zero.  The bubble family is uniformly bounded in $H^k$, so the Sobolev
inequality and \eqref{eq:tangent-ball-hardy} show that replacing
$f_\varepsilon$ by $f_\varepsilon^\sharp$ changes the $\dot H^k$ energy, the
squared $L^q$ norm, and the weighted $L^2$ integral by
$o(\varepsilon^\gamma)$.  Each $f_\varepsilon^\sharp$ is supported a positive
distance from the singular axis.  We henceforth write $f_\varepsilon$ for
this smooth family.

Combining \eqref{eq:local-energy-expansion},
\eqref{eq:local-weight-expansion}, and the conformal identities gives
\[
 Q_\lambda(\mathscr C f_\varepsilon)
 -S_{N,k}\|\mathscr C f_\varepsilon\|_q^2
 \le
 \bigl(\mathcal A_{N,k}^{\mathrm{loc}}
 -\lambda\mathcal F_{n,m,k}\bigr)\varepsilon^\gamma
 +o(\varepsilon^\gamma).
\]
The coefficient is negative precisely when
$\lambda>\Lambda_{\mathrm{loc}}$.
\end{proof}

A sufficient condition for the interval
$(\Lambda_{\mathrm{loc}},\Lambda_0)$ to be nonempty can be obtained without
numerical integration.  Since both hypergeometric factors in
\eqref{eq:F-local} are at
least one and $(1+t)^{-2k}\ge2^{-2k}$,
Beta integration and the Gamma duplication formula give
\begin{equation}
 \mathcal F_{n,m,k}
 \ge
 2^{-2k}\pi^{N/2}
 \frac{\Gamma(2k+1)}{\Gamma(N/2+2k+1)}.
 \label{eq:F-local-lower}
\end{equation}
Consequently,
\begin{equation}
 \Lambda_0>\Theta_{N,k}
 :=
 \frac{2^{4k+1}\Gamma(k+1)^2\Gamma(N/2+2k+1)}
 {(N-2k)\Gamma(N/2)\Gamma(2k+1)}
 \quad\Longrightarrow\quad
 \Lambda_{\mathrm{loc}}<\Lambda_0.
 \label{eq:theta-sufficient}
\end{equation}

\begin{theorem}[Existence below the spectral bottom]
\label{thm:subthreshold-existence}
Let
\[
 n\ge2,\qquad m\ge1,\qquad N=n+m,\qquad
 1\le k<\frac N2.
\]
Assume $\Lambda_0>0$ and $0<\lambda<\Lambda_0$.  The equation
\[
 P_kU-\lambda U=|U|^{q-2}U
\]
has a nontrivial weak solution
$U\in\Hlam=H^k(M)$, with the equation understood in $\Hlam^*$, in either of
the following cases:
\begin{enumerate}
 \item $N\ge4k$;
 \item $2k+2\le N<4k$ and
       $\Lambda_{\mathrm{loc}}<\lambda<\Lambda_0$.
\end{enumerate}
The second interval is nonempty whenever \eqref{eq:theta-sufficient} holds.
\end{theorem}

\begin{proof}
Lemma~\ref{lem:strict-high-dim} or
Proposition~\ref{prop:strict-low-dim} gives $\Slam<S_{N,k}$ in the
respective range.  The conclusion follows from
Theorem~\ref{thm:subthreshold-compactness}.
\end{proof}

\begin{remark}
The proposition treats $2k+2\le N<4k$.  Its argument does not cover the
remaining dimension $N=2k+1$, and no subthreshold existence conclusion is
asserted there.
\end{remark}

\section{The spectral threshold}

\begin{proposition}[Positive threshold embedding]
\label{prop:positive-threshold}
Assume
\[
 \Lambda_0>0,
 \qquad
 N\ge2k+2.
\]
Then there exists \(c=c(n,m,k)>0\) such that
\[
 Q_{\Lambda_0}(u)
 \ge c\|u\|_{L^q(M)}^2
 \qquad
 (u\in C_c^\infty(M)).
\]
In particular, \(\Sth>0\), and the embeddings
\[
 \Hcal\hookrightarrow L^q(M),
 \qquad
 \Ecal\hookrightarrow L^q(\R^N)
\]
are continuous.
\end{proposition}

\begin{proof}
For \(k\ge2\), apply the polyharmonic Hardy--Sobolev--Maz'ya inequality
\cite[Theorem~1.2 and Corollary~5.3]{Yang2021PolyharmonicHSM} with
\[
 n_Y=n-1,\qquad m_Y=m+1.
\]
At the critical exponent its weight is zero, while the Hardy coefficient in
that result is
\[
 c_{m+1,k}
 =
 \min_{\ell\ge0}
 \prod_{j=1}^k
 \left(
  \ell+\frac{m-1}{2}+k+1-2j
 \right)^2.
\]
Equation~\eqref{eq:spectral-bottom}, with \(\rho=0\), identifies this
coefficient with \(\Lambda_0\).
For \(k=1\), the corresponding inequality is
\cite[Section~2.1.7, Corollary~3, Eq.~(2.1.37)]
{Mazya2011SobolevSpaces}, where the Hardy coefficient is
\((m-1)^2/4\).  Equation~\eqref{eq:spectral-bottom} for \(k=1\) identifies it
with \(\Lambda_0\).  The assumption \(\Lambda_0>0\) excludes the
zero-coefficient case.  Setting $\lambda=\Lambda_0$ in
Equations~\eqref{eq:conformal-energy} and
\eqref{eq:conformal-lq} transfers these inequalities to \(M\).
\end{proof}

\begin{theorem}[Threshold extremal]\label{thm:main}
Let
\[
 n\ge2,\qquad m\ge1,\qquad N=n+m,\qquad
 1\le k<\frac N2.
\]
Assume
\[
 \Lambda_0>0,
 \qquad
 N\ge2k+2,
 \qquad
 \Sth<S_{N,k}.
\]
Then \(\Sth\) is attained by a nonzero \(u\in\Hcal\).  After the
normalization \(\|u\|_{L^q(M)}=1\),
\[
 Q_{\Lambda_0}(u)=\Sth
\]
and
\begin{equation}\label{eq:normalized-el}
 Q_{\Lambda_0}(u,\phi)
 =
 \Sth
 \int_M|u|^{q-2}u\phi\,\dd V
 \qquad
 (\phi\in\Hcal).
\end{equation}
Consequently,
\[
 U=\Sth^{1/(q-2)}u
\]
is a nontrivial weak solution of
\[
 P_kU-\Lambda_0U=|U|^{q-2}U
\]
in \(\Hcal^*\).
\end{theorem}

\section{Cocompactness after spherical-harmonic decomposition}

Recall from \eqref{eq:spherical-symbol} that the multiplier of \(P_k\) on
the degree-\(\ell\) spherical eigenspace is
\begin{equation}\label{eq:spherical-symbol-recalled}
 p_{k,\ell}(\rho)
 =
 \prod_{r=1}^k
 \left[
  \rho^2+
  \bigl(\ell+b_S+k+1-2r\bigr)^2
 \right].
\end{equation}
Let
\[
 \mathscr H_\ell(\Ss^m)
 =
 \ker\!\left(
  -\Delta_{\Ss^m}-\ell(\ell+m-1)
 \right),
 \qquad
 d_\ell=\dim\mathscr H_\ell(\Ss^m),
\]
and choose an orthonormal basis
\(\{Y_{\ell,a}\}_{a=1}^{d_\ell}\), with
\(Y_{0,1}=|\Ss^m|^{-1/2}\).  The corresponding coefficients of
\(u\) are
\[
 u_{\ell,a}(h)
 =
 \int_{\Ss^m}
 u(h,\theta)\overline{Y_{\ell,a}(\theta)}\,\dd\theta,
 \qquad
 u(h,\theta)
 =
 \sum_{\ell=0}^\infty\sum_{a=1}^{d_\ell}
 u_{\ell,a}(h)Y_{\ell,a}(\theta).
\]
Define
\[
 \mathfrak L_0
 =
 \{\ell\in\mathbb N_0:p_{k,\ell}(0)=\Lambda_0\},
 \qquad
\mathscr Y_0
 =
 \bigoplus_{\ell\in\mathfrak L_0}\mathscr H_\ell(\Ss^m).
\]
Proposition~\ref{prop:scattering-spectrum} shows that the standing
assumption \(\Lambda_0>0\) implies
\[
 \mathfrak L_0=\{0\},
 \qquad
 \mathscr Y_0=\mathscr H_0(\Ss^m)=\operatorname{span}\{1\}.
\]
Let \(P_{\mathscr Y_0}\) be the orthogonal projection in
\(L^2(\Ss^m)\) onto \(\mathscr Y_0\), and set
\[
 \Pi_0=I_{L^2(\Hh^n)}\otimes P_{\mathscr Y_0},
 \qquad
 \Pi_\perp=I-\Pi_0.
\]
Only \(\mathscr Y_0\) is finite dimensional; the product space
\(L^2(\Hh^n)\otimes\mathscr Y_0\) is not.

\begin{lemma}[Multiplier bounds for the spectral-bottom spherical subspace]
\label{lem:multiplier-split}
Assume \(\Lambda_0>0\).  There are constants \(c,C,c_+>0\), depending only
on \(n,m,k\), such that
\begin{align}
 c\rho^2(1+\rho^2)^{k-1}
 &\le p_{k,0}(\rho)-\Lambda_0
 \le C\rho^2(1+\rho^2)^{k-1},
 \label{eq:bottom-multiplier}\\
 p_{k,\ell}(\rho)-\Lambda_0
 &\ge c_+(1+\rho^2+\ell^2)^k,
 &&\ell\ge1.
 \label{eq:nonbottom-multiplier}
\end{align}
Consequently,
\[
 Q_{\Lambda_0}(\Pi_0u)
 \asymp
 \|u_{0,1}\|_{X_k(\Hh^n)}^2
\]
and
\[
 \|\Pi_\perp u\|_{H^k(M)}^2
 \le C Q_{\Lambda_0}(\Pi_\perp u),
\]
where $X_k(\Hh^n)$ is the completion of $C_c^\infty(\Hh^n)$ under the norm
\[
 \|f\|_{X_k}^2
 =
 \int_0^\infty\!\int_{\Ss^{n-1}}
 \rho^2(1+\rho^2)^{k-1}
 |\widehat f(\rho,\omega)|^2
 |\mathbf c(\rho)|^{-2}\,\dd\omega\,\dd\rho.
\]
Here \(\widehat f\) denotes the Helgason--Fourier transform on \(\Hh^n\),
and \(\mathbf c\) is the Harish--Chandra \(c\)-function.
\end{lemma}

\begin{proof}
Write
\[
 a_{0,r}=b_S+k+1-2r,
 \qquad 1\le r\le k.
\]
The assumption \(\Lambda_0>0\) implies \(a_{0,r}\ne0\) for every
\(1\le r\le k\).  Expanding the product in
\eqref{eq:spherical-symbol-recalled} gives
\begin{equation}\label{eq:bottom-polynomial-expansion}
\begin{aligned}
 p_{k,0}(\rho)-p_{k,0}(0)
 &=
 \prod_{r=1}^k(\rho^2+a_{0,r}^2)
 -\prod_{r=1}^ka_{0,r}^2\\
 &=
 \rho^2\sum_{s=1}^k c_s\rho^{2(s-1)},
\end{aligned}
\end{equation}
where every coefficient \(c_s\) is a positive elementary symmetric
polynomial in \(a_{0,1}^2,\ldots,a_{0,k}^2\).  Hence
\[
 0<
 \min_{1\le s\le k}c_s
 \le
 \max_{1\le s\le k}c_s
 <\infty.
\]
Together with
\[
 \sum_{s=1}^k\rho^{2(s-1)}
 \asymp(1+\rho^2)^{k-1},
\]
this proves \eqref{eq:bottom-multiplier}.

For \(\ell\ge1\), first fix \(L\) so large that
\[
 |a_{\ell,r}|\ge\frac{\ell}{2}
 \qquad(\ell\ge L,\ 1\le r\le k).
\]
For such \(\ell\),
\[
 p_{k,\ell}(\rho)
 \ge
 \prod_{r=1}^k\left(\rho^2+\frac{\ell^2}{4}\right)
 \ge
 4^{-k}(\rho^2+\ell^2)^k.
\]
Increasing \(L\) if necessary makes the last expression at least
\(2\Lambda_0\), uniformly in \(\rho\ge0\).  Hence
\[
 p_{k,\ell}(\rho)-\Lambda_0
 \ge c(1+\rho^2+\ell^2)^k
 \qquad(\ell\ge L).
\]
For the finitely many \(1\le\ell<L\), the quotient
\[
 \frac{p_{k,\ell}(\rho)-\Lambda_0}
 {(1+\rho^2+\ell^2)^k}
\]
is continuous and strictly positive on \([0,\infty)\), and tends to one as
\(\rho\to\infty\).  Taking its minimum over this finite set proves
\eqref{eq:nonbottom-multiplier}.  The two norm statements now follow by
Plancherel's theorem for the spherical and Helgason--Fourier decompositions.
\end{proof}

\begin{remark}
The space \(X_k(\Hh^n)\) does not control the full \(L^2\)-norm at
\(\rho=0\).  This is why the constant spherical coefficient \(u_{0,1}\)
is treated separately from the orthogonal spherical-harmonic complement.
\end{remark}

The operator $A$ is nonnegative and self-adjoint on $L^2(\Hh^n)$, with
spectrum $[0,\infty)$.  Let $E_A$ denote its projection-valued spectral
measure.  Thus, for every Borel set $I\subset[0,\infty)$,
\[
 E_A(I)=\mathbf 1_I(A),
 \qquad
 A=\int_{[0,\infty)}\rho\,\dd E_A(\rho).
\]
Under the Helgason--Fourier transform, these operators are the spectral
multipliers
\[
 \widehat{E_A(I)f}(\rho,\omega)
 =
 \mathbf 1_I(\rho)\widehat f(\rho,\omega).
\]
In particular, $E_A(I)$ is an orthogonal projection on $L^2(\Hh^n)$.  The
multiplier formula is norm-decreasing on the dense subspace
$C_c^\infty(\Hh^n)\subset X_k(\Hh^n)$, and hence $E_A(I)$ extends to a
contraction on $X_k(\Hh^n)$.
For \(0<\varepsilon<R\), set
\[
 E_{\le\varepsilon}:=E_A([0,\varepsilon]),
 \qquad
 E_{\varepsilon,R}:=E_A([\varepsilon,R]),
 \qquad
 E_{>R}:=E_A((R,\infty)).
\]
Since $-\Delta_{\Hh^n}=A^2+a_H^2$, the projection $E_{\le\varepsilon}$
selects the part of the spectrum of $-\Delta_{\Hh^n}$ lying in
$[a_H^2,a_H^2+\varepsilon^2]$.

\begin{lemma}[Low-frequency estimate at the spectral bottom]
\label{lem:low-frequency}
For every \(2<s<\infty\), there is \(C=C(n,s)\) such that, for every
$f\in X_k(\Hh^n)$,
\[
 \|E_{\le\varepsilon}f\|_{L^s(\Hh^n)}
 \le
 C\varepsilon^{1/2}\|f\|_{X_k(\Hh^n)}
 \qquad
 (0<\varepsilon\le1).
\]
\end{lemma}

\begin{proof}
By the definition of $X_k(\Hh^n)$ as a completion, it suffices first to
consider $f\in C_c^\infty(\Hh^n)$.  The spectral representation of $A$ gives
\begin{equation}
 \|Af\|_2^2
 =
 \int_0^\infty\!\int_{\Ss^{n-1}}
 \rho^2|\widehat f(\rho,\omega)|^2
 |\mathbf c(\rho)|^{-2}\,\dd\omega\,\dd\rho
 \le
 \|f\|_{X_k}^2,
 \label{eq:Af-controlled-by-Xk}
\end{equation}
where we used $k\ge1$ and hence $(1+\rho^2)^{k-1}\ge1$.

Let $s'=s/(s-1)$.  Then $1<s'<2$.  On hyperbolic space the spectral measure
is absolutely continuous for $\rho>0$; write
\[
 dE_A(\rho)=\mathsf e_A(\rho)\,\dd\rho.
\]
Chen and Hassell prove, in the present notation, that
\[
 \|\mathsf e_A(\rho)\|_{L^{s'}(\Hh^n)\to L^s(\Hh^n)}
 \le C_{n,s}\rho^2,
 \qquad 0<\rho\le1.
\]
See \cite[Theorem~1.6, Eq.~(1.19)]
{ChenHassell2018SpectralMeasure}.  Real hyperbolic space is nontrapping and
has no resonance at the spectral bottom, so the cited estimate applies.  For
\(0<\delta<\varepsilon\), put
\[
 T_{\delta,\varepsilon}
 =
 A^{-1}E_A([\delta,\varepsilon]).
\]
The cutoff at $\delta>0$ makes $A^{-1}$ bounded on the selected spectral
subspace.  Thus $T_{\delta,\varepsilon}$ is a bounded self-adjoint operator on
$L^2$.  Since $A^{-1}$ commutes with every spectral projection of $A$, the
spectral theorem gives
\[
 T_{\delta,\varepsilon}^2
 =
 A^{-2}E_A([\delta,\varepsilon])
 =
 \int_\delta^\varepsilon\rho^{-2}\mathsf e_A(\rho)\,\dd\rho.
\]
Therefore, for $g\in L^{s'}(\Hh^n)$, the preceding spectral-measure estimate
yields
\begin{align*}
 \|T_{\delta,\varepsilon}^2g\|_s
 &\le
 \int_\delta^\varepsilon
 \rho^{-2}\|\mathsf e_A(\rho)g\|_s\,\dd\rho\\
 &\le
 C_{n,s}\int_\delta^\varepsilon
 \rho^{-2}\rho^2\,\dd\rho\,\|g\|_{s'}\\
 &\le
 C_{n,s}(\varepsilon-\delta)\|g\|_{s'}.
\end{align*}
Consequently,
\[
 \|T_{\delta,\varepsilon}^2\|_{L^{s'}\to L^s}
 \le C(\varepsilon-\delta).
\]
For $g\in L^{s'}(\Hh^n)\cap L^2(\Hh^n)$, self-adjointness and Hölder's
inequality give
\[
 \|T_{\delta,\varepsilon}g\|_2^2
 =
 \bigl\langle
 T_{\delta,\varepsilon}^2g,g
 \bigr\rangle
 \le
 C(\varepsilon-\delta)\|g\|_{s'}^2.
\]
By density, $T_{\delta,\varepsilon}$ therefore extends from $L^{s'}$ to
$L^2$ with norm at most $C(\varepsilon-\delta)^{1/2}$.  If $h\in L^2$ and
$g\in L^{s'}\cap L^2$, self-adjointness gives
\[
 |\langle T_{\delta,\varepsilon}h,g\rangle|
 =
 |\langle h,T_{\delta,\varepsilon}g\rangle|
 \le
 C(\varepsilon-\delta)^{1/2}\|h\|_2\|g\|_{s'}.
\]
The dual characterization of the $L^s$ norm now gives
\begin{equation}
 \|A^{-1}E_A([\delta,\varepsilon])h\|_{L^s}
 \le C\varepsilon^{1/2}\|h\|_2
 \qquad(h\in L^2(\Hh^n)).
 \label{eq:regularized-low-frequency}
\end{equation}

We remove the lower cutoff next.  If
$0<\delta_1<\delta_2<\varepsilon$, applying the same argument to the interval
$[\delta_1,\delta_2]$ gives
\[
 \|T_{\delta_1,\varepsilon}h-T_{\delta_2,\varepsilon}h\|_s
 =
 \|A^{-1}E_A([\delta_1,\delta_2))h\|_s
 \le C(\delta_2-\delta_1)^{1/2}\|h\|_2.
\]
Thus $A^{-1}E_A([\delta,\varepsilon])h$ is Cauchy in $L^s$ as
$\delta\downarrow0$.  Since the spectral bottom is not an $L^2$ eigenvalue,
$E_A(\{0\})=0$, and the limit is denoted by
$A^{-1}E_{\le\varepsilon}h$.  Passing to the limit in
\eqref{eq:regularized-low-frequency} gives
\[
 \|A^{-1}E_{\le\varepsilon}h\|_{L^s}
 \le C\varepsilon^{1/2}\|h\|_2.
\]
Finally, take $h=Af$.  For every $\delta>0$, spectral multipliers commute and
give
\[
 A^{-1}E_A([\delta,\varepsilon])Af
 =E_A([\delta,\varepsilon])f.
\]
As $\delta\downarrow0$, the left-hand side converges in $L^s$ by the preceding
construction, while the right-hand side converges in $L^2$ to
$E_{\le\varepsilon}f$ because $E_A(\{0\})=0$.  The two limits agree as
distributions.  Combining this identity with
\eqref{eq:Af-controlled-by-Xk} proves the estimate for smooth $f$.  The
estimate then extends by density to the contraction $E_{\le\varepsilon}$ on
$X_k(\Hh^n)$.
\end{proof}

The parabolic isometries used below act on \(M\) by
\[
 g_{b,\sigma}(x,r,\theta)
 =
 (\sigma(x-b),\sigma r,\theta),
 \qquad
 b\in\R^{n-1},\quad\sigma>0.
\]
When these maps act on functions on \(\Hh^n\), the \(\theta\)-variable is
suppressed.

\begin{proposition}[Cocompactness for the spectral-bottom hyperbolic coefficient]
\label{prop:bottom-cocompactness}
Let \(f_j\) be bounded in \(X_k(\Hh^n)\).  Suppose that
\[
 f_j\circ g_{b_j,\sigma_j}
 \rightharpoonup0
 \quad\text{in }X_k(\Hh^n)
\]
for every parameter sequence \((b_j,\sigma_j)\).  Then
\[
 f_j\longrightarrow0
 \quad\text{in }L^q(\Hh^n).
\]
\end{proposition}

\begin{proof}
Set
\[
 \alpha_H=n\left(\frac12-\frac1q\right)
 =\frac{nk}{N}<k.
\]
For \(0<\varepsilon<1<R\), write
\[
 f_j
 =
 E_{\le\varepsilon}f_j
 +E_{\varepsilon,R}f_j
 +E_{>R}f_j.
\]
Lemma~\ref{lem:low-frequency} gives
\[
 \|E_{\le\varepsilon}f_j\|_q
 \le C\varepsilon^{1/2}\sup_j\|f_j\|_{X_k}.
\]
For the high-frequency term, we establish the following Bessel-potential
Sobolev estimate, initially for \(h\in C_c^\infty(\Hh^n)\).
Since
\[
 0<\alpha_H=\frac{nk}{N}<\frac n2,
 \qquad
 \frac1q=\frac12-\frac{\alpha_H}{n},
\]
the Euclidean fractional Sobolev inequality, transferred through uniform
normal coordinates, gives
\[
 \|\chi_i h\|_{L^q(\Hh^n)}
 \le
 C\|(\chi_i h)\circ\kappa_i^{-1}\|_{H^{\alpha_H}(\R^n)}.
\]
Here \(\{B_i\}\) is a uniformly locally finite cover by fixed-radius
geodesic balls, the normal coordinate chart \(\kappa_i\) is defined on the
concentric ball of twice that radius, and \(\chi_i\) is supported in \(B_i\).
Thus \((\chi_i h)\circ\kappa_i^{-1}\) may be extended by zero to \(\R^n\).
The functions \(\chi_i\) have uniformly bounded derivatives and satisfy
\(\sum_i\chi_i^2=1\).  For every integer \(r\ge0\), the Leibniz
rule, uniform metric bounds, and finite overlap give the localization
equivalence
\[
 \sum_i
 \|(\chi_i h)\circ\kappa_i^{-1}\|_{H^r(\R^n)}^2
 \asymp
 \|h\|_{H^r(\Hh^n)}^2.
\]
Interpolation between the two adjacent integer orders containing
\(\alpha_H\) therefore yields
\[
 \sum_i
 \|(\chi_i h)\circ\kappa_i^{-1}\|_{H^{\alpha_H}(\R^n)}^2
 \le
 C\|(1-\Delta_{\Hh^n})^{\alpha_H/2}h\|_2^2.
\]
Because \(q>2\), finite overlap and the embedding
\(\ell^2\hookrightarrow\ell^q\) imply
\begin{align*}
 \|h\|_q^2
 &\le
 C\left(\sum_i\|\chi_i h\|_q^q\right)^{2/q}\\
 &\le
 C\sum_i
 \|(\chi_i h)\circ\kappa_i^{-1}\|_{H^{\alpha_H}(\R^n)}^2.
\end{align*}
Finally,
\[
 1-\Delta_{\Hh^n}=1+A^2+a_H^2,
\]
so its \(\alpha_H/2\)-power has a norm equivalent to that of
\((1+A^2)^{\alpha_H/2}\).  We have proved
\begin{equation}\label{eq:hyperbolic-bessel-sobolev}
 \|h\|_{L^q(\Hh^n)}
 \le
 C\|(1+A^2)^{\alpha_H/2}h\|_2.
\end{equation}
Density extends this estimate to the spectral Sobolev space defined by the
norm on the right-hand side.
Indeed, Plancherel's theorem and the fact that \(\rho>R>1\) give
\begin{align*}
 \|(1+A^2)^{\alpha_H/2}E_{>R}f_j\|_2^2
 &=
 \int_R^\infty\!\int_{\Ss^{n-1}}
 (1+\rho^2)^{\alpha_H}
 |\widehat f_j(\rho,\omega)|^2
 |\mathbf c(\rho)|^{-2}\,\dd\omega\,\dd\rho,\\
 \frac{(1+\rho^2)^{\alpha_H}}
 {\rho^2(1+\rho^2)^{k-1}}
 &\le
 C\rho^{-2(k-\alpha_H)}
 \le
 CR^{-2(k-\alpha_H)}.
\end{align*}
Combining these two formulas yields
\begin{equation}\label{eq:bottom-high-frequency}
 \|E_{>R}f_j\|_q
 \le
 CR^{-(k-\alpha_H)}\|f_j\|_{X_k}.
\end{equation}

We finally treat the fixed band \([\varepsilon,R]\).  Put
\[
 h_j=E_{\varepsilon,R}f_j.
\]
On this band,
\[
 \|h_j\|_{H^k(\Hh^n)}
 \asymp_{\varepsilon,R}
 \|h_j\|_{X_k(\Hh^n)}.
\]
The spectral projection commutes with every hyperbolic isometry; hence
\[
 h_j\circ g_{b_j,\sigma_j}
 =
 E_{\varepsilon,R}(f_j\circ g_{b_j,\sigma_j})
 \rightharpoonup0
 \quad\text{in }H^k(\Hh^n)
\]
for every parameter sequence.  Suppose that
\(\|h_j\|_q\not\to0\).  The uniform local estimate
\eqref{eq:uniform-local-sobolev}, with \(M\) replaced by \(\Hh^n\),
provides \(r>0\), \(x_j\in\Hh^n\), and \(c_0>0\) such that
\[
 \int_{B(x_j,r)}|h_j|^q\,\dd V_{\Hh^n}\ge c_0.
\]
Choose \(g_j=g_{b_j,\sigma_j}\) with \(g_j(o)=x_j\).  Since
\[
 k>\alpha_H=n\left(\frac12-\frac1q\right),
\]
the local embedding is compact:
\[
 H^k(B(o,2r))\Subset L^q(B(o,r)).
\]
The weak convergence of \(h_j\circ g_j\) in \(H^k\) therefore implies
\[
 \int_{B(o,r)}|h_j\circ g_j|^q\,\dd V_{\Hh^n}\longrightarrow0,
\]
contrary to the preceding lower bound.  Thus
\[
 \|E_{\varepsilon,R}f_j\|_q\longrightarrow0.
\]
Thus
\[
 \limsup_{j\to\infty}\|f_j\|_q
 \le
 C\varepsilon^{1/2}+CR^{-(k-\alpha_H)}.
\]
Letting \(\varepsilon\downarrow0\) and \(R\to\infty\) proves the claim.
\end{proof}

\begin{lemma}[Euclidean critical inverse lemma]
\label{lem:euclidean-inverse}
Let \(F_j\) be bounded in \(H^k(\R^N)\), supported in a fixed compact
set, and suppose
\[
 F_j\rightharpoonup0
 \quad\text{in }H^k(\R^N),
 \qquad
 \limsup_j\|F_j\|_{L^q}>0.
\]
Then, after passing to a subsequence, there are
\(z_j\in\R^N\), \(\kappa_j\to\infty\), and
\(0\ne\varphi\in\dot H^k(\R^N)\) such that
\[
 \kappa_j^{-\gamma}
 F_j(z_j+Z/\kappa_j)
 \rightharpoonup\varphi
 \quad\text{in }\dot H^k(\R^N).
\]
\end{lemma}

\begin{proof}
Let \(\Delta_\nu\) be a homogeneous Littlewood--Paley projection to the
annulus \(|\xi|\asymp2^\nu\).  The refined Euclidean Sobolev inequality
\cite{MeyerGerardOru1997RefinedSobolev} gives
\[
 \|F\|_q
 \le
 C\|F\|_{\dot H^k}^{2/q}
 \left(
  \sup_{\nu\in\mathbb Z}
  2^{-\nu\gamma}\|\Delta_\nu F\|_\infty
 \right)^{1-2/q}.
\]
Hence there are \(\nu_j\) and \(z_j\) such that, with
\(\kappa_j=2^{\nu_j}\),
 \begin{equation}\label{eq:refined-sobolev-detection}
 \kappa_j^{-\gamma}|(\Delta_{\nu_j}F_j)(z_j)|\ge c>0.
 \end{equation}
Write
\[
 G_j(Z)=\kappa_j^{-\gamma}F_j(z_j+Z/\kappa_j).
\]
The sequence \(G_j\) is bounded in \(\dot H^k\).  After rescaling,
\(\Delta_{\nu_j}\) becomes a fixed annular projection \(\Delta_0\), and
the preceding lower bound is \(|(\Delta_0G_j)(0)|\ge c\).  Point evaluation
after the fixed projection \(\Delta_0\) is a continuous linear functional
on \(\dot H^k\).  Hence any weak limit \(\varphi\) satisfies
\(|(\Delta_0\varphi)(0)|\ge c\), and in particular \(\varphi\ne0\).

If \(\kappa_j\to0\), Bernstein's inequality and the uniform \(L^2\)
bound imply
\[
 \kappa_j^{-\gamma}\|\Delta_{\nu_j}F_j\|_\infty
 \le
 C\kappa_j^{-\gamma}\kappa_j^{N/2}\|F_j\|_2
 =
 C\kappa_j^k\|F_j\|_2
 \longrightarrow0,
\]
a contradiction.  Suppose next that
\[
 0<c\le\kappa_j\le C.
\]
Since \(\kappa_j=2^{\nu_j}\), passage to a subsequence makes
\(\nu_j=\nu_0\) independent of \(j\).  Let \(K_{\nu_0}\) be the Schwartz
kernel of \(\Delta_{\nu_0}\), and let \(K_0\) be a compact set containing
all supports of \(F_j\).  If \(|z_j|\to\infty\), then
\[
 |(\Delta_{\nu_0}F_j)(z_j)|
 \le
 \|F_j\|_2
 \|K_{\nu_0}(z_j-\cdot)\|_{L^2(K_0)}
 \longrightarrow0.
\]
If \(z_j\) is bounded, pass to \(z_j\to z_\infty\).  The map
\[
 F\longmapsto(\Delta_{\nu_0}F)(z_\infty)
\]
is a continuous linear functional on \(H^k(\R^N)\), and the continuity of
translations of \(K_{\nu_0}\) in \(L^2(K_0)\) gives
\[
 (\Delta_{\nu_0}F_j)(z_j)
 -
 (\Delta_{\nu_0}F_j)(z_\infty)
 \longrightarrow0.
\]
The weak convergence \(F_j\rightharpoonup0\) makes the second term tend to
zero.  Both alternatives contradict the lower bound in
\eqref{eq:refined-sobolev-detection}.  Therefore
\(\kappa_j\to\infty\).
\end{proof}

Let
\[
 \mathcal A=\R^{n-1}\times\{0\}\subset\R^N
\]
be the singular axis of the weight \(|y|^{-2k}\).  Define the
axis-preserving critical transformations
\[
 (\mathcal T_{b,\sigma}w)(x,y)
 =
 \sigma^\gamma w(\sigma(x-b),\sigma y),
 \qquad
 b\in\R^{n-1},\quad\sigma>0.
\]
They are unitary on \(\Ecal\), isometric on \(L^q\), and
\[
 \mathscr C(\mathcal T_{b,\sigma}w)
 =
 (\mathscr Cw)\circ g_{b,\sigma}.
\]
We write
\[
 G_{\mathcal A}
 =
 \{\mathcal T_{b,\sigma}:b\in\R^{n-1},\ \sigma>0\}
\]
for this group.  Every element of \(G_{\mathcal A}\) maps
\(\mathcal A\) to itself.  Conversely, under \(\mathscr C\) it is exactly
the critical action induced by the parabolic hyperbolic isometries
\(g_{b,\sigma}\).

For \(\zeta=(\xi,\eta)\in\R^{n-1}\times\R^{m+1}\) and \(\mu>0\), set
\[
 (\mathcal B_{\zeta,\mu}\varphi)(z)
 =
 \mu^\gamma\varphi(\mu(z-\zeta)).
\]
For a sequence
\[
 B_j=\mathcal B_{(\xi_j,\eta_j),\mu_j},
\]
the concentration scale is \(s_j=\mu_j^{-1}\), whereas the distance of
its center from \(\mathcal A\) is
\[
 d_j
 =
 \operatorname{dist}((\xi_j,\eta_j),\mathcal A)
 =
 |\eta_j|.
\]
We call \((B_j)\) a Euclidean concentration family escaping
\(\mathcal A\) if
\begin{equation}\label{eq:axis-escaping-packet}
 \frac{d_j}{s_j}
 =
 \mu_j|\eta_j|
 \longrightarrow\infty.
\end{equation}
Equivalently, after applying \(B_j^{-1}\), the distance of the rescaled
singular axis from the origin tends to infinity.
We denote by \(\langle\cdot,\cdot\rangle_{\Ecal}\) and
\(\langle\cdot,\cdot\rangle_{\dot H^k}\) the polarizations of the
corresponding squared norms.  For \(\chi\in C_c^\infty(\R^N)\), define
\begin{equation}\label{eq:packet-functional}
 L_{j,\chi}(w)
 =
 \langle w,B_j\chi\rangle_{\Ecal}
 +
 \Lambda_0
 \int_{\R^N}\frac{wB_j\chi}{|y|^{2k}}\,\dd z.
\end{equation}
For smooth functions this is
\[
 L_{j,\chi}(w)
 =
 \langle B_j^{-1}w,\chi\rangle_{\dot H^k}.
\]
A bounded sequence \(w_j\subset\Ecal\) is said to have no Euclidean profile
escaping \(\mathcal A\) if
\begin{equation}\label{eq:no-axis-escaping-profile}
 L_{j,\chi}(w_j)\longrightarrow0
\end{equation}
for every family satisfying \eqref{eq:axis-escaping-packet} and every
\(\chi\in C_c^\infty(\R^N)\).  This is a local testing condition; it does
not assert that \(B_j^{-1}w_j\) is globally bounded in \(\dot H^k\).

\begin{figure}[t]
\centering
\begin{tikzpicture}[x=1.05cm,y=1.05cm,font=\small]
 \draw[->] (-.2,0)--(5.2,0) node[right] {\(x\)};
 \node[left] at (-.2,0) {\(\mathcal A\)};
 \draw[thick,blue] (2,0.65) circle (0.48);
 \draw[<->] (2.65,0)--(2.65,0.65);
 \node[right] at (2.65,0.33) {\(d_j\)};
 \draw[->] (2,0.65)--(2.45,0.82);
 \node[above left] at (2.38,0.9) {\(s_j\)};
 \node[align=center] at (2.5,-0.75)
 {\(\displaystyle d_j/s_j=O(1)\)\\axis remains visible};

 \begin{scope}[xshift=7.2cm]
  \draw[->] (-.2,0)--(5.2,0) node[right] {\(x\)};
  \node[left] at (-.2,0) {\(\mathcal A\)};
  \draw[thick,red] (2.3,1.65) circle (0.25);
  \draw[<->] (2.8,0)--(2.8,1.65);
  \node[right] at (2.8,0.82) {\(d_j\)};
  \draw[->] (2.3,1.65)--(2.53,1.75);
  \node[above left] at (2.48,1.84) {\(s_j\)};
  \node[align=center] at (2.5,-0.75)
  {\(\displaystyle d_j/s_j\to\infty\)\\rescaled distance to axis diverges};
 \end{scope}
\end{tikzpicture}
\caption{The two concentration regimes relative to the singular axis
\(\mathcal A\).  If the ratio of the distance to \(\mathcal A\) and the
concentration scale remains bounded, the rescaled limit is represented by
the axis-preserving group \(G_{\mathcal A}\).  If this ratio tends to
infinity, the rescaled distance to the singular axis diverges, and the
limiting quadratic form is the Euclidean \(\dot H^k\)-energy.}
\label{fig:two-concentration-regimes}
\end{figure}
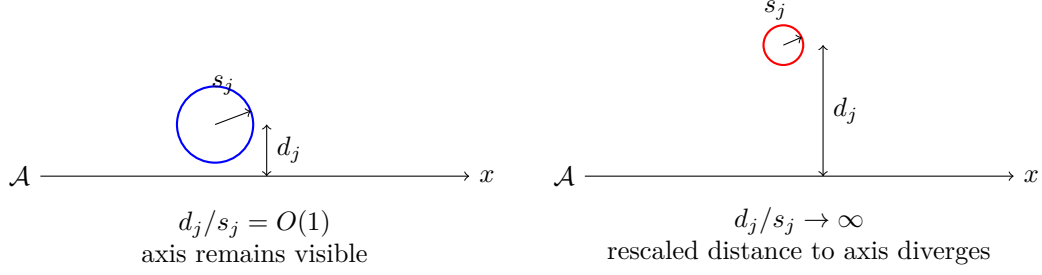

\begin{theorem}[Threshold cocompactness after spherical-harmonic decomposition]
\label{thm:modewise-cocompactness}
Assume \(\Lambda_0>0\) and \(\Sth>0\).  Let \(u_j\) be bounded in
\(\Hcal\).  Suppose
\[
 u_j\circ g_{b_j,\sigma_j}
 \rightharpoonup0
 \quad\text{in }\Hcal
\]
for every parameter sequence, and suppose that
 \(\mathscr C^{-1}u_j\) has no Euclidean profile escaping
 \(\mathcal A\).  Then
\[
 u_j\longrightarrow0
 \quad\text{in }L^q(M).
\]
\end{theorem}

\begin{proof}
The projections \(\Pi_0\) and \(\Pi_\perp\) commute with the parabolic
action.  The constant spherical coefficient \(u_{j,0,1}\) is bounded in
\(X_k(\Hh^n)\) and satisfies the hypothesis of
Proposition~\ref{prop:bottom-cocompactness}.  Hence
\[
 \|u_{j,0,1}\|_{L^q(\Hh^n)}\longrightarrow0.
\]
\begin{align}
 \|\Pi_0u_j\|_{L^q(M)}
 &\le
 \|Y_{0,1}\|_{L^q(\Ss^m)}
 \|u_{j,0,1}\|_{L^q(\Hh^n)}
 \longrightarrow0.
 \label{eq:bottom-spherical-lq-vanishing}
\end{align}

Set
\[
 v_j=\Pi_\perp u_j,
 \qquad
 w_j=\mathscr C^{-1}v_j.
\]
The orthogonality of the spherical decomposition and
\eqref{eq:nonbottom-multiplier} imply
\[
 \|v_j\|_{H^k(M)}^2
 \le
 C Q_{\Lambda_0}(v_j)
 \le
 C Q_{\Lambda_0}(u_j).
\]
In particular, \((v_j)\) is bounded in \(H^k(M)\).  Moreover,
\eqref{eq:conformal-energy} gives
\begin{equation}\label{eq:nonbottom-euclidean-hk-bound}
 \|w_j\|_{\dot H^k(\R^N)}^2
 =
 Q_{\Lambda_0}(v_j)
 +\Lambda_0\|v_j\|_{L^2(M)}^2
 \le C.
\end{equation}

Suppose that \(\|v_j\|_{L^q(M)}\not\to0\).  Applying
\eqref{eq:uniform-local-sobolev} to \(v_j\), there are \(r_0>0\),
\(\varepsilon_0>0\), and points
\[
 x_j=(h_j,\theta_j)\in\Hh^n\times\Ss^m
\]
such that
\begin{equation}\label{eq:nonbottom-local-mass}
 \int_{B(x_j,r_0)}|v_j|^q\,\dd V\ge\varepsilon_0.
\end{equation}
Choose a parabolic isometry \(g_j=g_{b_j,\sigma_j}\) with
\(g_j(o)=h_j\), and put
\[
 \widetilde v_j=v_j\circ g_j,
 \qquad
 \widetilde w_j
 =
 \mathcal T_{b_j,\sigma_j}w_j
 =
 \mathscr C^{-1}\widetilde v_j.
\]
After passing to a subsequence, \(\theta_j\to\theta_\infty\).  The
hypothesis of the theorem and \eqref{eq:nonbottom-multiplier} give
\[
 \widetilde v_j\rightharpoonup0
 \quad\text{in }H^k(M).
\]

Choose an open set
\(\Omega\Subset\R^N\setminus\mathcal A\) whose image under \(\mathscr C\)
contains a fixed product neighborhood of \((o,\theta_\infty)\), and choose
\(0\le\zeta\in C_c^\infty(\Omega)\) equal to one on a smaller neighborhood
containing the balls in \eqref{eq:nonbottom-local-mass} after recentering.
Then
\[
 F_j=\zeta\widetilde w_j
\]
is supported in a fixed compact subset of \(\Omega\), is bounded in
\(H^k(\R^N)\) by \eqref{eq:nonbottom-euclidean-hk-bound} and local
equivalence of the weighted and unweighted norms, and satisfies
\[
 F_j\rightharpoonup0\quad\text{in }H^k(\R^N),
 \qquad
 \limsup_{j\to\infty}\|F_j\|_{L^q(\R^N)}>0.
\]
Lemma~\ref{lem:euclidean-inverse} therefore provides
\[
 \widetilde B_j
 =
 \mathcal B_{\widetilde\zeta_j,\widetilde\mu_j},
 \qquad
 \widetilde\mu_j\longrightarrow\infty,
 \qquad
 \widetilde B_j^{-1}F_j
 \rightharpoonup\varphi\ne0
 \quad\text{in }\dot H^k(\R^N).
\]
The lower bound \eqref{eq:refined-sobolev-detection} and the rapid decay of
the Littlewood--Paley kernel imply
\[
 \operatorname{dist}
 (\widetilde\zeta_j,\operatorname{supp}\zeta)\longrightarrow0.
\]
Since \(\operatorname{supp}\zeta\Subset\R^N\setminus\mathcal A\), there
are constants \(c,C>0\) such that, after extraction,
\begin{equation}\label{eq:recentered-packet-away-axis}
 c\le|\widetilde\eta_j|\le C,
 \qquad
 \widetilde\zeta_j=(\widetilde\xi_j,\widetilde\eta_j).
\end{equation}
Furthermore,
\[
 \widetilde B_j^{-1}F_j(Z)
 =
 \zeta\!\left(
  \widetilde\zeta_j+\frac{Z}{\widetilde\mu_j}
 \right)
 \widetilde B_j^{-1}\widetilde w_j(Z).
\]
After passing to a subsequence, assume that
\(\widetilde\zeta_j\to\widetilde\zeta_\infty\).  Put
\[
 c_\zeta=\zeta(\widetilde\zeta_\infty),
 \qquad
 a_j(Z)=\zeta\!\left(
  \widetilde\zeta_j+\frac{Z}{\widetilde\mu_j}
 \right).
\]
Then \(a_j\to c_\zeta\) in \(C^k_{\mathrm{loc}}\).  The sequence
\(\widetilde B_j^{-1}\widetilde w_j\) is bounded in \(\dot H^k\) by
\eqref{eq:nonbottom-euclidean-hk-bound}.  If \(c_\zeta=0\), the Leibniz rule
and the local \(C^k\)-convergence of \(a_j\) would give
\[
 a_j\widetilde B_j^{-1}\widetilde w_j
 \rightharpoonup0
 \quad\text{in }\dot H^k_{\mathrm{loc}},
\]
contrary to
\(a_j\widetilde B_j^{-1}\widetilde w_j
=\widetilde B_j^{-1}F_j\rightharpoonup\varphi\ne0\).
Thus \(c_\zeta>0\).  Moreover, for every compactly supported test function,
the same Leibniz calculation gives
\[
 \widetilde B_j^{-1}\widetilde w_j
 \rightharpoonup c_\zeta^{-1}\varphi
 \quad\text{in }\dot H^k_{\mathrm{loc}}.
\]
Consequently there exists \(\chi\in C_c^\infty(\R^N)\) such that
\begin{equation}\label{eq:recentered-nonzero-packet-functional}
 \limsup_{j\to\infty}
 \left|
 \left\langle
  \widetilde B_j^{-1}\widetilde w_j,\chi
 \right\rangle_{\dot H^k}
 \right|>0.
\end{equation}

In the original coordinates, define
\[
 B_j
 =
 \mathcal T_{b_j,\sigma_j}^{-1}\widetilde B_j
 =
 \mathcal B_{\zeta_j,\mu_j}.
\]
A direct composition of the two critical transformations gives
\begin{equation}\label{eq:packet-composition}
 \mu_j=\frac{\widetilde\mu_j}{\sigma_j},
 \qquad
 \zeta_j
 =
 (-\sigma_jb_j,0)+\sigma_j\widetilde\zeta_j,
 \qquad
 B_j^{-1}w_j
 =
 \widetilde B_j^{-1}\widetilde w_j.
\end{equation}
In particular,
\begin{equation}\label{eq:packet-axis-ratio-invariant}
 \mu_j|(\zeta_j)_y|
 =
 \widetilde\mu_j|\widetilde\eta_j|
 \longrightarrow\infty
\end{equation}
by \eqref{eq:recentered-packet-away-axis}.  Thus \((B_j)\) is a Euclidean
concentration family escaping \(\mathcal A\), and
\eqref{eq:recentered-nonzero-packet-functional} shows that
\(B_j^{-1}w_j\rightharpoonup c_\zeta^{-1}\varphi\ne0\) locally in
\(\dot H^k\); equivalently, its pairing with the test function in that
display does not vanish.

Finally, set
\[
 w_j^{(0)}=\mathscr C^{-1}(\Pi_0u_j).
\]
For every \(\chi\in C_c^\infty(\R^N)\), the identity defining
\(L_{j,\chi}\), Hölder's inequality, and
\eqref{eq:bottom-spherical-lq-vanishing} give
\begin{align}
 |L_{j,\chi}(w_j^{(0)})|
 &=
 \left|
 \int_{\R^N}
 B_j^{-1}w_j^{(0)}(-\Delta)^k\chi\,\dd Z
 \right| \notag\\
 &\le
 \|w_j^{(0)}\|_{L^q(\R^N)}
 \|(-\Delta)^k\chi\|_{L^{q'}(\R^N)}
 \longrightarrow0.
 \label{eq:bottom-packet-functional-vanishing}
\end{align}
Together with \eqref{eq:bottom-packet-functional-vanishing},
\eqref{eq:recentered-nonzero-packet-functional} implies that
\(\mathscr C^{-1}u_j=w_j+w_j^{(0)}\) does not satisfy
\eqref{eq:no-axis-escaping-profile}, contrary to the hypothesis.  Hence
\(\Pi_\perp u_j\to0\) in \(L^q(M)\).  Together with
\eqref{eq:bottom-spherical-lq-vanishing}, this proves the theorem.
\end{proof}

\section{Profile decomposition at the spectral threshold}

\subsection{Euclidean profiles escaping the singular axis}

\begin{lemma}[Extraction at Euclidean scales far from the singular axis]
\label{lem:packet-to-profile}
Let \(w_j\) be bounded in \(\Ecal\) and in \(L^q(\R^N)\).  Let
\[
 B_j=\mathcal B_{(\xi_j,\eta_j),\mu_j},
 \qquad
 \mu_j|\eta_j|\to\infty.
\]
If, for some \(\chi_0\in C_c^\infty(\R^N)\),
\[
 \limsup_j|L_{j,\chi_0}(w_j)|>0,
\]
then, after passing to a subsequence, there is
\[
 0\ne\varphi\in\dot H^k(\R^N)\cap L^q(\R^N)
\]
such that
\[
 L_{j,\chi}(w_j)
 \longrightarrow
 \langle\varphi,\chi\rangle_{\dot H^k}
\]
for every \(\chi\in C_c^\infty(\R^N)\).  After a further extraction,
\[
 B_j^{-1}w_j\longrightarrow\varphi
\quad\text{locally almost everywhere}.
\]
\end{lemma}

\begin{proof}
Put \(c_j=\mu_j\eta_j\) and \(q'=q/(q-1)\).  For every fixed compactly supported
\(\chi\),
\begin{align}
 \|B_j\chi\|_{\Ecal}^2
 &=
 \|\chi\|_{\dot H^k}^2
 -
 \Lambda_0
 \int_{\R^N}
 \frac{|\chi(X,Y)|^2}{|Y+c_j|^{2k}}\,\dd X\,\dd Y
 \longrightarrow
 \|\chi\|_{\dot H^k}^2,
 \label{eq:packet-energy-limit}\\
 \left\|\frac{B_j\chi}{|y|^{2k}}\right\|_{q'}
 &=
 \left\|
 \frac{\chi(X,Y)}{|Y+c_j|^{2k}}
 \right\|_{q'}
 \longrightarrow0.
 \label{eq:packet-potential-limit}
\end{align}
The threshold pairing, Hölder's inequality, and the continuous embedding
\(\Ecal\hookrightarrow L^q\) therefore give
\[
 \limsup_j|L_{j,\chi}(w_j)|
 \le C\|\chi\|_{\dot H^k}.
\]
A diagonal subsequence argument on a countable dense subset of
\(\dot H^k\) yields a bounded linear functional on \(\dot H^k\).  The Riesz
theorem gives
\(\varphi\in\dot H^k\).  The hypothesis involving $\chi_0$ implies that
\(\varphi\ne0\).

Set \(V_j=B_j^{-1}w_j\).  The \(L^q\)-norm is invariant, so, after
extraction, \(V_j\rightharpoonup V\) in \(L^q\).  For compactly
supported \(\chi\),
\[
 L_{j,\chi}(w_j)
 =
 \int_{\R^N}V_j(-\Delta)^k\chi.
\]
It follows that
\[
 (-\Delta)^k(V-\varphi)=0
\]
in distributions.  The homogeneous Sobolev inequality puts
\(\varphi\) in \(L^q\).  The difference \(V-\varphi\) is a tempered
polyharmonic distribution.  Its Fourier transform is supported at the
origin, so it is a polynomial; membership in \(L^q(\R^N)\) excludes every
nonzero polynomial.  Thus
\[
 V=\varphi\in L^q.
\]

Let \(K\Subset K'\Subset\R^N\), and write
\(H^{-k}(K')=(H_0^k(K'))^*\).  For large \(j\), the shifted
singular plane \(Y=-c_j\) lies outside \(K'\).  If
\(\chi\in H_0^k(K')\), density and the functional bound above give
\[
 \bigl|\langle(-\Delta)^kV_j,\chi\rangle\bigr|
 =|L_{j,\chi}(w_j)|
 \le C\|\chi\|_{H^k(K')}.
\]
Thus \((-\Delta)^kV_j\) is uniformly bounded in \(H^{-k}(K')\);
the global \(L^q\)-bound gives a local \(L^2\)-bound.  The interior
elliptic estimate yields
\[
 \|V_j\|_{H^k(K)}
 \le
 C_K\left(
  \|(-\Delta)^kV_j\|_{H^{-k}(K')}
  +\|V_j\|_{L^2(K')}
 \right).
\]
Rellich compactness and a diagonal argument give local strong convergence
below the critical exponent and hence local almost-everywhere convergence.
\end{proof}

\subsection{Profiles relative to the axis-preserving group}

Two parameter families in \(G_{\mathcal A}\),
\(\mathcal T_j^{(\iota)}
=\mathcal T_{b_j^{(\iota)},\sigma_j^{(\iota)}}\) and
\(\mathcal T_j^{(\jmath)}
=\mathcal T_{b_j^{(\jmath)},\sigma_j^{(\jmath)}}\) are called orthogonal if
\begin{equation}\label{eq:parameter-orthogonality}
 \left|
 \log\frac{\sigma_j^{(\iota)}}{\sigma_j^{(\jmath)}}
 \right|
 +
 \sigma_j^{(\iota)}
 |b_j^{(\iota)}-b_j^{(\jmath)}|
 \longrightarrow\infty.
\end{equation}
Indeed,
\[
 (\mathcal T_j^{(\iota)})^{-1}\mathcal T_j^{(\jmath)}
 =
 \mathcal T_{
  \sigma_j^{(\iota)}(b_j^{(\jmath)}-b_j^{(\iota)}),
  \,\sigma_j^{(\jmath)}/\sigma_j^{(\iota)}
 }.
\]
If \eqref{eq:parameter-orthogonality} fails, the relative
transformations have a strongly convergent subsequence; if it holds, they
converge to zero in the weak operator topology on \(\Ecal\).

For a bounded sequence \(r=(r_j)\subset\Ecal\), define its maximal weak
defect relative to \(G_{\mathcal A}\) by
\begin{equation}\label{eq:maximal-axis-profile-defect}
 \mathfrak a_{\mathcal A}(r)
 =
 \sup\left\{
  \|\psi\|_{\Ecal}:
  \begin{array}{l}
   \text{there exist \(j_s\to\infty\) and
   \(T_s\in G_{\mathcal A}\) such that}\\
   T_s^{-1}r_{j_s}\rightharpoonup\psi
   \text{ in }\Ecal
  \end{array}
 \right\}.
\end{equation}

\begin{theorem}[Profile decomposition relative to the singular axis]
\label{thm:cylindrical-profile}
Assume \(\Sth>0\).  Let \(w_j\) be bounded in \(\Ecal\) and suppose that
it has no Euclidean profile escaping \(\mathcal A\).  After passing to a
subsequence, there are
\[
 w^{(0)},\psi^{(1)},\psi^{(2)},\ldots\in\Ecal,
\]
pairwise orthogonal parameter families in \(G_{\mathcal A}\),
\[
 \mathcal T_j^{(\nu)}
 =
 \mathcal T_{b_j^{(\nu)},\sigma_j^{(\nu)}},
\]
and remainders \(r_j^{(J)}\) such that
\[
 w_j
 =
 w^{(0)}
 +\sum_{\nu=1}^J
 \mathcal T_j^{(\nu)}\psi^{(\nu)}
 +r_j^{(J)}.
\]
For each fixed \(J\),
\begin{align}
 \|w_j\|_{\Ecal}^2
 &=
 \|w^{(0)}\|_{\Ecal}^2
 +\sum_{\nu=1}^J\|\psi^{(\nu)}\|_{\Ecal}^2
 +\|r_j^{(J)}\|_{\Ecal}^2
 +o(1),
 \label{eq:energy-profile-split}\\
 \|w_j\|_q^q
 &=
 \|w^{(0)}\|_q^q
 +\sum_{\nu=1}^J\|\psi^{(\nu)}\|_q^q
 +\|r_j^{(J)}\|_q^q
 +o(1).
 \label{eq:lq-profile-split}
\end{align}
For every fixed \(J\), the remainder \(r_j^{(J)}\) has no Euclidean profile
escaping \(\mathcal A\), and
\begin{equation}\label{eq:maximal-profile-defect-vanishing}
 \mathfrak a_{\mathcal A}(r^{(J)})\longrightarrow0
 \qquad(J\to\infty).
\end{equation}
Finally,
\begin{equation}\label{eq:profile-remainder}
 \lim_{J\to\infty}\limsup_{j\to\infty}
 \|r_j^{(J)}\|_q=0.
\end{equation}
\end{theorem}

\begin{proof}
After passing to a subsequence, let \(w^{(0)}\) be the weak limit of \(w_j\),
and set \(r_j^{(0)}=w_j-w^{(0)}\).  If
\(\mathfrak a_{\mathcal A}(r^{(J)})=0\) at a finite stage, then
\[
 T_j^{-1}r_j^{(J)}\rightharpoonup0
 \quad\text{in }\Ecal
\]
for every sequence \(T_j\in G_{\mathcal A}\).  The subtraction of finitely
many transported profiles preserves the absence of Euclidean profiles
escaping \(\mathcal A\), as follows from
\eqref{eq:axis-profile-packet-incompatibility} and the density argument
following it.  Hence
Theorem~\ref{thm:modewise-cocompactness}, applied after the conformal
transfer, gives
\[
 r_j^{(J)}\longrightarrow0
 \quad\text{in }L^q(\R^N),
\]
and the extraction terminates.  Otherwise choose the next profile so that
\[
 \|\psi^{(J+1)}\|_{\Ecal}
 \ge\frac12\mathfrak a_{\mathcal A}(r^{(J)}),
\]
where, for a suitable
\(T_j^{(J+1)}
=\mathcal T_{b_j^{(J+1)},\sigma_j^{(J+1)}}\),
\[
 (T_j^{(J+1)})^{-1}r_j^{(J)}
 \rightharpoonup\psi^{(J+1)}
 \quad\text{in }\Ecal.
\]
Define
\begin{equation}\label{eq:profile-remainder-recursion}
 r_j^{(J+1)}
 =
 r_j^{(J)}-T_j^{(J+1)}\psi^{(J+1)}.
\end{equation}
A later parameter family must be orthogonal to every earlier one.  If not,
the relative transformations would have a strongly convergent subsequence,
and composition with its limit would produce a nonzero weak limit of an
earlier remainder in the earlier coordinate system, contrary to the weak
convergence that follows from
\eqref{eq:profile-remainder-recursion}.

The one-step Hilbert identity gives, for fixed \(J\),
\[
 \|w_j\|_{\Ecal}^2
 =
 \|w^{(0)}\|_{\Ecal}^2
 +\sum_{\nu=1}^J\|\psi^{(\nu)}\|_{\Ecal}^2
 +\|r_j^{(J)}\|_{\Ecal}^2
 +o(1).
\]
In particular,
\[
 \sum_{\nu=1}^\infty\|\psi^{(\nu)}\|_{\Ecal}^2<\infty,
 \qquad
 \mathfrak a_{\mathcal A}(r^{(J)})\longrightarrow0.
\]
Local coercivity away from the singular plane and local compactness give,
after a diagonal extraction, almost-everywhere convergence in each extracted
coordinate.  In particular,
\[
 (T_j^{(J+1)})^{-1}r_j^{(J)}\longrightarrow\psi^{(J+1)}
 \quad\text{a.e. on }\R^N,
\]
and \eqref{eq:profile-remainder-recursion} gives
\[
 (T_j^{(J+1)})^{-1}r_j^{(J+1)}\longrightarrow0
 \quad\text{a.e. on }\R^N.
\]
The Brézis--Lieb lemma \cite[Theorem~1]{BrezisLieb1983}, together with the
\(L^q\)-isometry of \(T_j^{(J+1)}\), therefore yields the one-step identity
\[
 \|r_j^{(J)}\|_q^q
 =
 \|\psi^{(J+1)}\|_q^q
 +\|r_j^{(J+1)}\|_q^q
 +o(1).
\]
Applying the same argument first to the ordinary weak limit \(w^{(0)}\) and
then iterating this identity proves \eqref{eq:lq-profile-split}.

To prove \eqref{eq:profile-remainder}, observe first that every transported
\(G_{\mathcal A}\)-profile has asymptotically vanishing
pairing with each Euclidean concentration escaping \(\mathcal A\).  For
\(\psi\in C_c^\infty(\R^N\setminus\mathcal A)\),
\[
 B_j^{-1}\mathcal T_{b_j,\sigma_j}\psi(X,Y)
 =
 \left(\frac{\sigma_j}{\mu_j}\right)^\gamma
 \psi\left(
  \sigma_j(\xi_j-b_j)+\frac{\sigma_j}{\mu_j}X,
  \sigma_j\eta_j+\frac{\sigma_j}{\mu_j}Y
 \right).
\]
A critical dilation or translation of a fixed compactly supported smooth
function converges weakly to zero in \(\dot H^k_{\mathrm{loc}}\) whenever
its scale tends to zero or infinity, or whenever its translated center
leaves every compact set.  Applied to the preceding formula, this gives the
necessary implication
\begin{equation}\label{eq:axis-profile-packet-incompatibility}
\begin{aligned}
 B_j^{-1}\mathcal T_{b_j,\sigma_j}\psi
 \not\rightharpoonup0
 \quad\Longrightarrow\quad&
 \frac{\sigma_j}{\mu_j}\longrightarrow\tau\in(0,\infty),\\
 &
 \sigma_j(\xi_j-b_j)=O(1),
 \qquad
 \sigma_j\eta_j=O(1).
\end{aligned}
\end{equation}
The last condition and the convergence of the scale ratio imply
\[
 \mu_j|\eta_j|
 =
 \frac{|\sigma_j\eta_j|}{\sigma_j/\mu_j}
 =
 O(1),
\]
contrary to \eqref{eq:axis-escaping-packet}.  Hence subtraction of any fixed
finite family of \(G_{\mathcal A}\)-profiles preserves
\eqref{eq:no-axis-escaping-profile}.  By density, the same conclusion holds
for every \(\psi\in\Ecal\), because the functionals \(L_{j,\chi}\) are
uniformly bounded on \(\Ecal\) by
\eqref{eq:packet-energy-limit}--\eqref{eq:packet-potential-limit}.

Suppose now that \eqref{eq:profile-remainder} fails.  Then for some
\(\delta>0\) one can choose \(J_s\uparrow\infty\) and
\(j_s\uparrow\infty\) as follows.  Let
\(\{\Phi_\ell\}_{\ell\ge1}\) be dense in the unit ball of \(\Ecal\).
For every fixed \(J\) and \(\ell\), the definition of
\(\mathfrak a_{\mathcal A}\) gives
\begin{equation}\label{eq:uniform-profile-defect}
 \limsup_{j\to\infty}\sup_{T\in G_{\mathcal A}}
 \bigl|\langle T^{-1}r_j^{(J)},\Phi_\ell\rangle_{\Ecal}\bigr|
 \le \mathfrak a_{\mathcal A}(r^{(J)}).
\end{equation}
If \eqref{eq:uniform-profile-defect} failed, one could choose transformations
whose corresponding weak limit has norm greater than
\(\mathfrak a_{\mathcal A}(r^{(J)})\), contradicting its definition.  A
diagonal choice of \(J_s\uparrow\infty\) and \(j_s\uparrow\infty\) then
ensures that
\[
 z_s:=r_{j_s}^{(J_s)},
 \qquad
 \|z_s\|_q\ge\delta/2,
\]
and
\begin{equation}\label{eq:slow-diagonal-tests}
 \sup_{\substack{T\in G_{\mathcal A}\\1\le\ell\le s}}
 \bigl|\langle T^{-1}z_s,\Phi_\ell\rangle_{\Ecal}\bigr|
 \le \mathfrak a_{\mathcal A}(r^{(J_s)})+s^{-1}.
\end{equation}
The same choice makes the total energy cross-term error among the first
\(J_s\) profiles at most \(s^{-1}\).  Since
\(\mathfrak a_{\mathcal A}(r^{(J_s)})\to0\), density in the unit ball and
\eqref{eq:slow-diagonal-tests} show that \(z_s\) is weakly zero after every
sequence of transformations in \(G_{\mathcal A}\).  The energy identity
and the chosen
cross-term bound also give
\(\sup_s\|z_s\|_{\Ecal}<\infty\).

The sequence \(z_s\) also has no Euclidean profile escaping
\(\mathcal A\).  Fix \(L\).
Orthogonality and the choice of diagonal indices give
\[
 \left\|
  \sum_{\nu=L+1}^{J_s}
  \mathcal T_{j_s}^{(\nu)}\psi^{(\nu)}
 \right\|_{\Ecal}^2
 \le
 2\sum_{\nu>L}\|\psi^{(\nu)}\|_{\Ecal}^2+s^{-1}.
\]
Equations \eqref{eq:packet-energy-limit}--\eqref{eq:packet-potential-limit}
bound the pairing of this tail with each fixed test function \(\chi\) by
\[
 C_\chi
 \left(2\sum_{\nu>L}\|\psi^{(\nu)}\|_{\Ecal}^2\right)^{1/2}.
\]
Letting \(L\to\infty\) shows that
\(z_s\) has no Euclidean profile escaping \(\mathcal A\).  Applying
Theorem~\ref{thm:modewise-cocompactness} after the conformal transfer
gives \(\|z_s\|_q\to0\), a contradiction.
\end{proof}

\section{Compactness under the strict Euclidean inequality}

\begin{proposition}[Exclusion of Euclidean concentrations far from the singular axis]
\label{prop:packet-exclusion}
Assume
\[
 \Sth<S_{N,k}.
\]
Let \(w_j\in\Ecal\) be a normalized minimizing sequence:
\[
 \|w_j\|_q=1,
 \qquad
 \|w_j\|_{\Ecal}^2\longrightarrow\Sth.
\]
Then \(w_j\) has no Euclidean profile escaping \(\mathcal A\).
\end{proposition}

\begin{proof}
Suppose, to the contrary, that \(w_j\) has a Euclidean profile escaping
\(\mathcal A\).  By
Lemma~\ref{lem:packet-to-profile}, after extraction there is
\[
 0\ne\varphi\in\dot H^k(\R^N)\cap L^q(\R^N)
\]
which is the local almost-everywhere limit in the coordinates obtained by
applying \(B_j^{-1}\).  Set
\[
 a=\|\varphi\|_q^q\in(0,1],
 \qquad
 \beta_q=\frac2q\in(0,1).
\]
The Euclidean Sobolev inequality and the strict gap give
\[
 \|\varphi\|_{\dot H^k}^2
 +\Sth(1-a)^{\beta_q}
 \ge
 S_{N,k}a^{\beta_q}+\Sth(1-a)^{\beta_q}
 >
 \Sth.
\]
Choose \(\psi\in C_c^\infty(\R^N)\) close to \(\varphi\) in
\(\dot H^k\cap L^q\), and define
\[
 D_\psi
 =
 2\langle\varphi,\psi\rangle_{\dot H^k}
 -\|\psi\|_{\dot H^k}^2,
 \qquad
 c_\psi
 =
 \|\varphi\|_q^q-\|\varphi-\psi\|_q^q.
\]
The approximation may be chosen so that
\[
 0<c_\psi<1,
 \qquad
 D_\psi+\Sth(1-c_\psi)^{\beta_q}>\Sth.
\]

Let \(p_j=B_j\psi\) and \(r_j=w_j-p_j\).  Since
\(\mu_j|\eta_j|\to\infty\), the function \(p_j\) is eventually
supported away from the singular plane; hence \(p_j,r_j\in\Ecal\).
The limits \eqref{eq:packet-energy-limit}--\eqref{eq:packet-potential-limit}
give
\[
 \|w_j\|_{\Ecal}^2
 =
 \|r_j\|_{\Ecal}^2+D_\psi+o(1).
\]
Local almost-everywhere convergence and the Brézis--Lieb lemma give
\[
 \|r_j\|_q^q=1-c_\psi+o(1).
\]
Applying the threshold inequality
\(\|r_j\|_{\Ecal}^2\ge\Sth\|r_j\|_q^2\) contradicts the preceding
strict inequality.  Thus no Euclidean profile can escape
\(\mathcal A\).
\end{proof}

\begin{proposition}[Compactness modulo the axis-preserving group]
\label{prop:compactness-modulo-group}
Under the hypotheses of Theorem~\ref{thm:main}, let \(w_j\in\Ecal\) satisfy
\[
 \|w_j\|_q=1,
 \qquad
 \|w_j\|_{\Ecal}^2\longrightarrow\Sth.
\]
Then there exist a subsequence, transformations
\(\mathcal T_j\), and a nonzero $w\in\Ecal$ such that
\[
 \mathcal T_j^{-1}w_j
 \rightharpoonup w
 \quad\text{in }\Ecal.
\]
Equivalently, after the corresponding product-space recentering, the
minimizing sequence has a nonzero weak limit in \(\Hcal\).
\end{proposition}

\begin{proof}
By Proposition~\ref{prop:packet-exclusion}, the minimizing sequence satisfies
the hypotheses of Theorem~\ref{thm:cylindrical-profile}.  Taking first
\(j\to\infty\) in
\eqref{eq:lq-profile-split} and then \(J\to\infty\), using
\eqref{eq:profile-remainder}, gives
\[
 1
 =
 \|w^{(0)}\|_q^q
 +\sum_{\nu=1}^{\infty}\|\psi^{(\nu)}\|_q^q.
\]
Thus the ordinary weak limit or one of the \(G_{\mathcal A}\)-profiles has
positive \(L^q\)-mass.  Recentering along the corresponding parameter
family, or doing nothing for \(w^{(0)}\), preserves the energy and
\(L^q\)-norm and produces the asserted nonzero weak limit.
\end{proof}

\section{Attainment and the Euler--Lagrange equation}

\begin{proof}[Proof of Theorem~\ref{thm:main}]
Let \(u_j\in\Hcal\) satisfy
\[
 \|u_j\|_{L^q(M)}=1,
 \qquad
 Q_{\Lambda_0}(u_j)\longrightarrow\Sth.
\]
By
Proposition~\ref{prop:compactness-modulo-group}, after recentering and
passing to a subsequence,
\[
 u_j\rightharpoonup u\ne0
 \quad\text{in }\Hcal.
\]
In the conformal coordinates, the threshold form is locally coercive on
compact sets separated from the Euclidean axis.  The conformal image of
every compact set \(K\Subset M\) is separated from \(\mathcal A\).  Local
Rellich compactness therefore allows us, after passing to a subsequence, to
assume that \(u_j\to u\) almost everywhere locally.  Put \(v_j=u_j-u\).
The Hilbert identity and the
Brézis--Lieb lemma give
\begin{align*}
 Q_{\Lambda_0}(u_j)
 &=
 Q_{\Lambda_0}(u)+Q_{\Lambda_0}(v_j)+o(1),\\
 1
 &=
 \|u\|_q^q+\|v_j\|_q^q+o(1).
\end{align*}
After another subsequence, set
\[
 a=\|u\|_q^q>0,
 \qquad
 b=\lim_j\|v_j\|_q^q.
\]
Then \(a+b=1\).  By the definition of \(\Sth\),
\[
 \Sth
 \ge
 \Sth\bigl(a^{2/q}+b^{2/q}\bigr).
\]
Since \(2/q\in(0,1)\), strict concavity gives
\[
 a^{2/q}+b^{2/q}>1
\]
when \(a,b>0\).  Because \(a>0\), it follows that \(b=0\) and
\(a=1\).  Hence
\[
 \|u\|_q=1.
\]
The variational lower bound gives
\(Q_{\Lambda_0}(u)\ge\Sth\), while the energy splitting and the
nonnegativity of \(Q_{\Lambda_0}(v_j)\) give the reverse inequality.
Therefore \(Q_{\Lambda_0}(u)=\Sth\).

For \(\phi\in\Hcal\), differentiate the quotient at \(u\).  This gives
\[
 Q_{\Lambda_0}(u,\phi)
 =
 \Sth\int_M|u|^{q-2}u\phi\,\dd V.
\]
The right-hand side belongs to \(\Hcal^*\) by
Proposition~\ref{prop:positive-threshold}.  Finally, with
\[
 U=\Sth^{1/(q-2)}u,
\]
we obtain
\[
 Q_{\Lambda_0}(U,\phi)
 =
 \int_M|U|^{q-2}U\phi\,\dd V
\]
for every \(\phi\in\Hcal\).
\end{proof}

\appendix

\section{A scattering computation of the product GJMS operator}
\label{app:scattering}

We derive \eqref{eq:spherical-symbol} directly from scattering on
\(\Hh^{N+1}\).  Using the conformal-boundary realization of
\(\Hh^n\times\Ss^m\), we separate the scattering equation in
Helgason--Fourier variables and spherical harmonics and evaluate the
resulting hypergeometric connection coefficients.  This computation
recovers the Case--Malchiodi factorization and determines both the bottom of
the \(L^2\)-spectrum and the exact lower bound on the orthogonal complement
of the constant spherical functions.

Throughout the appendix,
\[
 a_H=\frac{n-1}{2},
 \qquad
 b_S=\frac{m-1}{2},
 \qquad
 A^2=-\Delta_{\Hh^n}-a_H^2,
\]
and
\[
 D_S
 =
 \sqrt{-\Delta_{\Ss^m}+b_S^2}.
\]
Thus, for a spherical harmonic \(Y_{\ell,a}\) of degree \(\ell\),
\[
 D_SY_{\ell,a}=(\ell+b_S)Y_{\ell,a}.
\]
The symbol \(P_k\) continues to denote the order-\(2k\) GJMS operator.

\subsection{The hyperbolic filling and the scattering normalization}

Use the hyperboloid model of \(\Hh^{N+1}\) in
\(\R^{1,n}\oplus\R^{m+1}\).  If
\[
 h\in\Hh^n,
 \qquad
 \theta\in\Ss^m,
 \qquad
 r\ge0,
\]
then the Fermi-coordinate map around the totally geodesic copy of
\(\Hh^n\) is
\[
 (r,h,\theta)
 \longmapsto
 (\cosh r\,h,\sinh r\,\theta).
\]
Differentiation, together with
\(\langle h,\dd h\rangle=0\) and
\(\langle\theta,\dd\theta\rangle=0\), gives
\begin{equation}\label{eq:app-fermi-metric}
 g_+
 =
 \dd r^2+\cosh^2r\,g_{\Hh^n}
 +\sinh^2r\,g_{\Ss^m}.
\end{equation}
Set
\[
 x=2e^{-r}.
\]
Since
\[
 \cosh r=x^{-1}\left(1+\frac{x^2}{4}\right),
 \qquad
 \sinh r=x^{-1}\left(1-\frac{x^2}{4}\right),
\]
equation \eqref{eq:app-fermi-metric} becomes
\begin{equation}\label{eq:app-pe-metric}
 g_+
 =
 x^{-2}\left[
 \dd x^2+
 \left(1+\frac{x^2}{4}\right)^2g_{\Hh^n}
 +
 \left(1-\frac{x^2}{4}\right)^2g_{\Ss^m}
 \right].
\end{equation}
It follows that \(g_{\Hh^n}+g_{\Ss^m}\) is the boundary representative on
the open chart
\[
 \Ss^N\setminus\Ss^{n-1}
 \cong
 \Hh^n\times\Ss^m.
\]
The compact conformal boundary of the filling is \(\Ss^N\).  A function in
\(C_c^\infty(M)\) is supported away from the omitted
\(\Ss^{n-1}\), so its boundary density extends smoothly by zero across that
set.  We may therefore use the Graham--Zworski residue normalization on
\(\Ss^N\) and then restrict the resulting local conformal operator to this
chart.

Let
\[
 s=\frac N2+\nu.
\]
For \(\operatorname{Re}\nu>0\) and \(\nu\notin\mathbb Z\), with \(\nu\)
initially chosen away from the poles of the resolvent and the Gamma
functions, consider the equation
\begin{equation}\label{eq:app-scattering-equation}
 \left(-\Delta_{g_+}-s(N-s)\right)U=0.
\end{equation}
The two indicial exponents are \(N-s=N/2-\nu\) and
\(s=N/2+\nu\).  With prescribed leading coefficient \(f\), the expansion
has the form
\[
 U
 =
 x^{N/2-\nu}(f+O(x^2))
 +
 x^{N/2+\nu}(h+O(x^2)),
\]
and the scattering operator is
\[
 S\left(\frac N2+\nu\right)f=h.
\]
Graham and Zworski
\cite[Theorem~1]{GrahamZworski2003Scattering} proved that
\begin{equation}\label{eq:app-gz-residue}
 \operatorname*{Res}_{s=N/2+k}S(s)
 =
 \frac{(-1)^{k+1}}
 {2^{2k}k!(k-1)!}\,P_k.
\end{equation}
In the present filling,
\[
 \sigma_{L^2(\Hh^{N+1})}(-\Delta_{g_+})
 =
 \left[\frac{N^2}{4},\infty\right),
\]
whereas
\[
 s(N-s)\big|_{s=N/2+k}
 =
 \frac{N^2}{4}-k^2.
\]
Thus the spectral exclusion required in the residue formula is automatic.
Equivalently, the pole in \(S\) is cancelled by
\(\Gamma(\nu)/\Gamma(-\nu)\), and
\begin{equation}\label{eq:app-normalized-scattering-limit}
 P_k
 =
 \lim_{\nu\to k}
 2^{2\nu}\frac{\Gamma(\nu)}{\Gamma(-\nu)}
 S\left(\frac N2+\nu\right).
\end{equation}

\subsection{Separation of variables}

For the metric \eqref{eq:app-fermi-metric}, the nonnegative Laplacian is
\[
 -\Delta_{g_+}
 =
 -\partial_r^2
 -(n\tanh r+m\coth r)\partial_r
 +\frac{-\Delta_{\Hh^n}}{\cosh^2r}
 +\frac{-\Delta_{\Ss^m}}{\sinh^2r}.
\]
Let
\[
 -\Delta_{\Hh^n}\phi_{\rho,\omega}
 =
 (a_H^2+\rho^2)\phi_{\rho,\omega},
 \qquad \rho\ge0,
\]
in the Helgason--Fourier resolution, and let
\[
 -\Delta_{\Ss^m}Y_{\ell,a}
 =
 \ell(\ell+m-1)Y_{\ell,a}.
\]
For
\[
 U(r,h,\theta)
 =
 R(r)\phi_{\rho,\omega}(h)Y_{\ell,a}(\theta),
\]
equation \eqref{eq:app-scattering-equation} is equivalent to
\begin{equation}\label{eq:app-radial-ode}
\begin{aligned}
 R''
 &+(n\tanh r+m\coth r)R'
 \\
 &+
 \left[
 s(N-s)
 -\frac{a_H^2+\rho^2}{\cosh^2r}
 -\frac{\ell(\ell+m-1)}{\sinh^2r}
 \right]R
 =0.
\end{aligned}
\end{equation}

Put
\[
 z=\tanh^2r,
 \qquad
 \beta=\ell+b_S.
\]
Then
\[
 \partial_r=2\sqrt z(1-z)\partial_z,
\]
and direct substitution in \eqref{eq:app-radial-ode} gives
\begin{equation}\label{eq:app-radial-z}
\begin{aligned}
 z(1-z)R_{zz}
 &+\frac{m+1+(n-3)z}{2}R_z
 \\
 &+
 \left[
 \frac{s(N-s)}{4(1-z)}
 -\frac{a_H^2+\rho^2}{4}
 -\frac{\ell(\ell+m-1)}{4z}
 \right]R
 =0.
\end{aligned}
\end{equation}
Regularity at \(r=0\), where \(z\sim r^2\), selects the behavior
\(R\sim z^{\ell/2}\).  At \(z=1\), the incoming boundary behavior is
\((1-z)^{(N-s)/2}\).  We therefore write
\[
 R(z)
 =
 z^{\ell/2}(1-z)^{(N-s)/2}F(z).
\]
Equation \eqref{eq:app-radial-z} reduces to
\begin{equation}\label{eq:app-hypergeometric-equation}
 z(1-z)F''
 +
 \left[c-(\alpha_++\alpha_-+1)z\right]F'
 -\alpha_+\alpha_-F
 =0,
\end{equation}
where
\begin{equation}\label{eq:app-hypergeometric-parameters}
 c=\beta+1,
 \qquad
 \alpha_\pm
 =
 \frac{\beta+1-\nu\pm i\rho}{2}.
\end{equation}
In particular,
\[
 \alpha_++\alpha_-=\beta+1-\nu,
 \qquad
 \alpha_+\alpha_-
 =
 \frac{(\beta+1-\nu)^2+\rho^2}{4},
 \qquad
 c-\alpha_+-\alpha_-=\nu.
\]
The solution regular at \(r=0\) is consequently
\[
 R(z)
 =
 z^{\ell/2}(1-z)^{(N-s)/2}
 {}_2F_1(\alpha_+,\alpha_-;c;z).
\]

The hypergeometric connection formula at \(z=1\) yields
\begin{align*}
 {}_2F_1(\alpha_+,\alpha_-;c;z)
 &={}
 C_0\,{}_2F_1(\alpha_+,\alpha_-;1-\nu;1-z)
 \\
 &\quad+
 C_1(1-z)^\nu
 {}_2F_1(c-\alpha_+,c-\alpha_-;1+\nu;1-z),
\end{align*}
where
\[
 C_0
 =
 \frac{\Gamma(c)\Gamma(\nu)}
 {\Gamma(c-\alpha_+)\Gamma(c-\alpha_-)},
 \qquad
 C_1
 =
 \frac{\Gamma(c)\Gamma(-\nu)}
 {\Gamma(\alpha_+)\Gamma(\alpha_-)}.
\]
Thus the first summand is \(C_0(1+O(1-z))\), whereas the second is
\(C_1(1-z)^\nu(1+O(1-z))\).  The connection formula holds initially when
the two exponents are distinct and extends meromorphically in \(\nu\).
Since
\[
 1-z
 =
 \operatorname{sech}^2r
 =
 \frac{x^2}{(1+x^2/4)^2},
\]
the two terms have leading powers \(x^{N/2-\nu}\) and
\(x^{N/2+\nu}\), with no additional constant.  Hence
\begin{equation}\label{eq:app-scattering-multiplier}
\begin{aligned}
 S\left(\frac N2+\nu\right)(\rho,\ell)
 &=
 \frac{\Gamma(-\nu)}{\Gamma(\nu)}
 \\
 &\quad\times
 \frac{
 \Gamma\!\left(\frac{\beta+1+\nu+i\rho}{2}\right)
 \Gamma\!\left(\frac{\beta+1+\nu-i\rho}{2}\right)
 }{
 \Gamma\!\left(\frac{\beta+1-\nu+i\rho}{2}\right)
 \Gamma\!\left(\frac{\beta+1-\nu-i\rho}{2}\right)
 }.
\end{aligned}
\end{equation}

\subsection{Integer parameters and the product factorization}

Multiplying \eqref{eq:app-scattering-multiplier} by the normalizing factor
in \eqref{eq:app-normalized-scattering-limit} gives the meromorphic family
\[
 \mathcal P_\nu(\rho,\ell)
 =
 4^\nu
 \frac{
 \Gamma\!\left(\frac{\beta+1+\nu+i\rho}{2}\right)
 \Gamma\!\left(\frac{\beta+1+\nu-i\rho}{2}\right)
 }{
 \Gamma\!\left(\frac{\beta+1-\nu+i\rho}{2}\right)
 \Gamma\!\left(\frac{\beta+1-\nu-i\rho}{2}\right)
 }.
\]
At \(\nu=k\), the identity
\[
 \frac{\Gamma(z+k)}{\Gamma(z)}
 =
 \prod_{i=0}^{k-1}(z+i)
\]
implies
\begin{align}
 \mathcal P_k(\rho,\ell)
 &=
 \prod_{i=0}^{k-1}
 \left[\rho^2+(\beta+1-k+2i)^2\right]
 \notag\\
 &=
 \prod_{j=1}^k
 \left[
 \rho^2+(\ell+b_S+k+1-2j)^2
 \right]
 =
 p_{k,\ell}(\rho).
 \label{eq:app-integer-multiplier}
\end{align}
By the joint functional calculus,
\begin{equation}\label{eq:app-operator-factorization}
 P_k
 =
 \prod_{j=1}^k
 \left[
 A^2+(D_S+k+1-2j)^2
 \right].
\end{equation}
Although \(D_S\) is written as a square root, the full product is local.
Indeed, the nonzero shifts occur in opposite pairs, and
\[
 \left[A^2+(D_S+r)^2\right]
 \left[A^2+(D_S-r)^2\right]
 =
 (A^2+D_S^2+r^2)^2-4r^2D_S^2.
\]
If \(k\) is odd, the remaining central factor is \(A^2+D_S^2\).
Thus \eqref{eq:app-operator-factorization} is a polynomial in \(A^2\) and
\(D_S^2\), hence a differential operator of order \(2k\).  Formula
\eqref{eq:app-integer-multiplier} agrees with
\cite[Corollary~3.3]{CaseMalchiodi2024GJMS}.

\subsection{The \texorpdfstring{\(L^2\)}{L2}-spectrum and the spherical threshold gap}

Set
\[
 a_\ell
 =
 p_{k,\ell}(0)
 =
 \prod_{j=1}^k
 (\ell+b_S+k+1-2j)^2.
\]
For fixed \(\ell\), the polynomial
\[
 x\longmapsto
 \prod_{j=1}^k
 \left[x+(\ell+b_S+k+1-2j)^2\right]
\]
is increasing on \([0,\infty)\), is nonconstant, and tends to infinity.
The spectrum of \(P_k\) on
\(L^2(\Hh^n)\otimes\mathscr H_\ell(\Ss^m)\) is therefore
\([a_\ell,\infty)\).  Consequently,
\begin{equation}\label{eq:app-full-spectrum}
 \sigma_{L^2(M)}(P_k)
 =
 [\Lambda_0,\infty),
 \qquad
 \Lambda_0=\min_{\ell\ge0}a_\ell.
\end{equation}
The endpoint \(\Lambda_0\) is not an \(L^2\)-eigenvalue: in every spherical
sector that attains the minimum, equality requires \(\rho=0\), and the
Helgason--Fourier Plancherel measure has no atom there.

Define
\[
 \pi_k(t)
 =
 \prod_{j=1}^k(t+k+1-2j).
\]
Then
\[
 a_\ell=\pi_k(b_S+\ell)^2,
 \qquad
 \pi_k(t+2)
 =
 \frac{t+k+1}{t-k+1}\pi_k(t).
\]
On each parity class, \(|\pi_k(t)|\) is therefore strictly increasing once
the initial value is nonzero; if the denominator vanishes, the preceding
value is already zero.  Hence
\begin{equation}\label{eq:app-general-spectral-bottom}
 \Lambda_0
 =
 \min_{\varepsilon\in\{0,1\}}
 \pi_k(b_S+\varepsilon)^2
 =
 4^k
 \min_{\varepsilon\in\{0,1\}}
 \left|
 \frac{
 \Gamma\!\left(\frac{m+2k+1+2\varepsilon}{4}\right)
 }{
 \Gamma\!\left(\frac{m-2k+1+2\varepsilon}{4}\right)
 }
 \right|^2.
\end{equation}
At a pole of the denominator, the Gamma quotient in
\eqref{eq:app-general-spectral-bottom} is understood through the polynomial
\(\pi_k\).  The roots of \(\pi_k\) are
\[
 1-k,\,3-k,\,\ldots,\,k-1.
\]
It follows that
\begin{equation}\label{eq:app-zero-bottom}
 \Lambda_0=0
 \quad\Longleftrightarrow\quad
 m\ \text{is odd and}\ k\ge\frac{m+1}{2}.
\end{equation}

\begin{proposition}[The positive spectral bottom and the spherical gap]
\label{prop:scattering-spectrum}
Assume \(\Lambda_0>0\).  Then
\[
 a_0<a_1<a_2<\cdots.
\]
In particular,
\[
 \mathfrak L_0=\{0\},
 \qquad
 \mathscr Y_0
 =
 \mathscr H_0(\Ss^m)
 =
 \operatorname{span}\{1\}.
\]
Moreover,
\begin{equation}\label{eq:app-lambda-zero-explicit}
 \Lambda_0
 =
 4^k
 \left|
 \frac{
 \Gamma\!\left(\frac{m+2k+1}{4}\right)
 }{
 \Gamma\!\left(\frac{m-2k+1}{4}\right)
 }
 \right|^2,
\end{equation}
and the bottom on the orthogonal spherical-harmonic complement is
\begin{equation}\label{eq:app-lambda-one-explicit}
 \Lambda_1
 =
 4^k
 \left|
 \frac{
 \Gamma\!\left(\frac{m+2k+3}{4}\right)
 }{
 \Gamma\!\left(\frac{m-2k+3}{4}\right)
 }
 \right|^2.
\end{equation}
Thus
\[
 \delta_{\Ss}:=\Lambda_1-\Lambda_0>0,
\]
and
\[
 p_{k,\ell}(\rho)-\Lambda_0
 \ge\delta_{\Ss}
 \qquad(\ell\ge1,\ \rho\ge0).
\]
\end{proposition}

\begin{proof}
Suppose first that \(m\) is odd, so that \(b_S\) is a nonnegative integer.
If \(b_S\le k-1\), choosing \(\ell=k-1-b_S\) produces the zero factor
\[
 \ell+b_S-k+1=0
\]
in \(a_\ell\), contrary to \(\Lambda_0>0\).  Hence \(b_S\ge k\).
Every factor in \(\pi_k(b_S+\ell)\) is then positive and strictly increasing
with \(\ell\), which proves the assertion in this case.

Suppose next that \(m\) is even.  Then
\[
 t=b_S+\ell
\]
is a positive half-integer.  If \(t>k-1\), every factor in \(\pi_k(t)\) is
positive and increases strictly when \(t\) is replaced by \(t+1\).  It remains
to consider \(t<k-1\).  Put
\[
 U_t=\frac{t+k+1}{2},
 \qquad
 C_t=\frac{k+1-t}{2}.
\]
The Gamma-product identity and the reflection formula give
\[
 |\pi_k(t)|
 =
 \frac{2^k}{\pi}
 \Gamma(U_t)\Gamma(C_t)|\sin(\pi C_t)|.
\]
For \(t+1\), the corresponding expression is
\[
 |\pi_k(t+1)|
 =
 \frac{2^k}{\pi}
 \Gamma\!\left(U_t+\frac12\right)
 \Gamma\!\left(C_t-\frac12\right)
 |\cos(\pi C_t)|.
\]
The number \(C_t\) is a quarter-integer, and hence
\[
 |\sin(\pi C_t)|=|\cos(\pi C_t)|.
\]
Define
\[
 G(x)=\frac{\Gamma(x+1/2)}{\Gamma(x)},
 \qquad x>0.
\]
If \(\psi=\Gamma'/\Gamma\), then
\[
 \frac{\dd}{\dd x}\log G(x)
 =
 \psi\!\left(x+\frac12\right)-\psi(x)>0,
\]
because
\[
 \psi'(x)=\sum_{j=0}^\infty\frac{1}{(x+j)^2}>0.
\]
Thus \(G\) is strictly increasing.  Since
\[
 U_t-\left(C_t-\frac12\right)=t+\frac12>0,
\]
we obtain
\[
 \frac{|\pi_k(t+1)|}{|\pi_k(t)|}
 =
 \frac{G(U_t)}{G(C_t-1/2)}
 >1.
\]
This proves strict monotonicity in the even-dimensional spherical case as
well.  Equations \eqref{eq:app-lambda-zero-explicit} and
\eqref{eq:app-lambda-one-explicit} now follow from
\eqref{eq:app-general-spectral-bottom}.
\end{proof}

\begin{remark}
The number \(\delta_{\Ss}\) is a gap between the bottoms of the constant
spherical sector and its orthogonal complement.  It is not an open gap above
\(\Lambda_0\) in the full \(L^2\)-spectrum, since the constant spherical
sector already contributes the entire interval \([\Lambda_0,\infty)\).
\end{remark}

\subsection{The cylindrical Hardy coefficient}

Return to the conformal coordinates
\[
 z=(x,y)\in
 \R^{n-1}\times(\R^{m+1}\setminus\{0\}),
 \qquad
 r=|y|.
\]
Let
\[
 \gamma=\frac{N-2k}{2},
 \qquad
 u=\mathscr Cw=r^\gamma w.
\]
The conformal covariance of \(P_k\) gives the pointwise identity
\[
 P_k(r^\gamma w)
 =
 r^{(N+2k)/2}(-\Delta_{\R^N})^kw.
\]
Since \(\dd V_M=r^{-N}\dd z\), it follows that
\begin{align}
 \int_M(P_ku)u\,\dd V
 &=
 \int_{\R^N}(-\Delta)^kw\,w\,\dd z,
 \label{eq:app-conformal-energy}\\
 \int_Mu^2\,\dd V
 &=
 \int_{\R^N}\frac{w^2}{|y|^{2k}}\,\dd z.
 \label{eq:app-conformal-l2}
\end{align}
Combining \eqref{eq:app-full-spectrum} with
\eqref{eq:app-conformal-energy}--\eqref{eq:app-conformal-l2} yields
\begin{equation}\label{eq:app-sharp-hardy}
 \int_{\R^N}(-\Delta)^kw\,w\,\dd z
 \ge
 \Lambda_0
 \int_{\R^N}\frac{w^2}{|y|^{2k}}\,\dd z
\end{equation}
for \(w\in C_c^\infty(\R^N\setminus\R^{n-1})\).

The coefficient in \eqref{eq:app-sharp-hardy} is sharp in this test class.
Indeed, choose an \(L^2\)-normalized hyperbolic spectral
function \(f_\varepsilon\) whose Helgason--Fourier transform is supported in
\(0<\rho<\varepsilon\), and tensor it with a normalized constant spherical
harmonic.  Its Rayleigh quotient for \(P_k\) tends to
\(p_{k,0}(0)=\Lambda_0\) as \(\varepsilon\downarrow0\).  The space
\(C_c^\infty(M)\) is a form core for the positive elliptic realization of
\(P_k\), so these functions can be approximated in the form norm by compactly
supported smooth functions.  The conformal transfer then gives a sequence in
\(C_c^\infty(\R^N\setminus\R^{n-1})\) whose quotient in
\eqref{eq:app-sharp-hardy} tends to \(\Lambda_0\).

With Yang's variables
\[
 n_Y=n-1,
 \qquad
 m_Y=m+1,
\]
the degree-\(\ell\) factorization \eqref{eq:app-integer-multiplier} is
\[
 \prod_{i=0}^{k-1}
 \left[
 -\Delta_{\Hh^n}-\frac{(n-1)^2}{4}
 +\frac{(m+1+2\ell-2k+4i)^2}{4}
 \right].
\]
Consequently,
\[
 c_{m+1,k}
 =
 \min_{\ell\ge0}
 \prod_{i=0}^{k-1}
 \frac{(m+1+2\ell-2k+4i)^2}{4}
 =
 \Lambda_0.
\]
For \(k\ge2\), the scattering calculation thus gives an independent
geometric derivation of the sharp Hardy coefficient and the pure Hardy
inequality that occur in Yang's polyharmonic Hardy--Sobolev--Maz'ya theorem
\cite[Theorem~1.2 and Corollary~5.3]{Yang2021PolyharmonicHSM}.  For \(k=1\),
the corresponding cylindrical Hardy inequality is the classical result of
Maz'ya used in Proposition~\ref{prop:positive-threshold}.  The scattering
calculation does not, by itself, prove the positive critical Sobolev
remainder
\[
 \|w\|_{\Ecal}^2\ge C\|w\|_{L^q(\R^N)}^2,
 \qquad C>0,
\]
which for \(k\ge2\) follows from Yang's theorem and for \(k=1\) follows from
Maz'ya's inequality.

\bibliographystyle{plainurl}
\bibliography{project}

\bigskip

\begin{flushleft}
\small
\textsc{Q. Hua}\\
School of Mathematical Sciences, University of Science and Technology of China\\
\emph{Email address:}
\href{mailto:qqhua@ustc.edu.cn}
{\nolinkurl{qqhua@ustc.edu.cn}}

\medskip

\textsc{J. Li}\\
School of Mathematical Sciences, University of Science and Technology of China\\
\emph{Email address:}
\href{mailto:jungangli@ustc.edu.cn}
{\nolinkurl{jungangli@ustc.edu.cn}}

\medskip

\textsc{C. Tao}\\
School of Mathematics and Statistics, Beijing Technology and Business University\\
\emph{Email address:}
\href{mailto:taochunxia01@126.com}
{\nolinkurl{taochunxia01@126.com}}
\end{flushleft}

\end{document}